\documentclass[12pt,a4paper]{article}       
\usepackage{a4wide}
\usepackage{mathptmx}      
\usepackage{amssymb,amsthm,amsmath,amsfonts}
\renewcommand{\Re}{\operatorname{Re}}
\renewcommand{\Im}{\operatorname{Im}}
\newcommand{\RR}{{\mathbb{R}}}
\newcommand{\CC}{{\mathbb{C}}}
\newtheorem{lemma}{Lemma}
\newtheorem{theorem}{Theorem}
\newtheorem{remark}{Remark}
\begin{document}

\title{Stability Estimates and Structural Spectral Properties of Saddle Point Problems\thanks{The research of the first and the third author was supported by the Austrian Science Fund (FWF): W1214-N15, project DK12.}%
}

\author{Wolfgang Krendl\thanks{Doctoral Program Computational Mathematics, Johannes Kepler University Linz, 4040 Linz,
        Austria,\newline \texttt{wolfgang.krendl@dk-compmath.jku.at}}
         \and
        Valeria Simoncini\thanks{Dipartimento di Matematica, Universit\`{a} di Bologna, Piazza di Porta S. Donato 5, 40127 Bologna,\newline Italy \texttt{valeria.simoncini@unibo.it}}
	\and
        Walter Zulehner\thanks{Institute of Computational Mathematics, Johannes Kepler University Linz, 4040 Linz, Austria,\newline \texttt{zulehner@numa.uni-linz.ac.at}}%
}

\date{}

\maketitle

\begin{abstract}
For a general class of saddle point problems sharp estimates for Babu\v{s}ka's inf-sup stability constants are derived in terms of the constants in Brezzi's theory.
In the finite-dimensional Hermitian case more detailed spectral properties of preconditioned saddle point matrices are presented, which are helpful for the convergence analysis of common Krylov subspace methods. 
The theoretical results are applied to two model problems from optimal control with time-periodic state equations. 
Numerical experiments with the preconditioned minimal residual method are reported.

\smallskip
\noindent{\bf Keywords:} Saddle point problems, Babu\v{s}ka-Brezzi theory, inf-sup constants, eigenvalues of saddle point matrices, optimal control, time-periodic state equation

\smallskip
\noindent{\bf Mathematical Subject Classification (2010):} 65F08, 65N22, 65K10, 49K40
\end{abstract}

\section{Introduction}
\label{intro}
In this paper we consider linear problems in saddle point form:
Find $u \in V$ and $p \in Q$ such that
\begin{equation} \label{M}
  \mathcal{M} 
  \begin{bmatrix} u \\ p \end{bmatrix}
   =
  \begin{bmatrix} f \\ g \end{bmatrix}
  \quad \text{with} \quad
  \mathcal{M} = 
  \begin{bmatrix}
     A & B^* \\
     B & -C
   \end{bmatrix}
\end{equation}
with complex Hilbert spaces $V$ and $Q$, $f \in V^*$ and $g \in Q^*$ and linear operators
\[
  A \in L(V,V^*), \quad
  B \in L(V,Q^*), \quad
  B^* \in L(Q,V^*), \quad
  C \in L(Q,Q^*). 
\]
Here $H^*$ denotes the dual of a Hilbert space $H$, which is defined as the set of all bounded and antilinear functionals on $H$, $L(H_1,H_2)$ denotes the set of all bounded and linear operators from $H_1$ to $H_2$, and $B^*$ is the adjoint of $B$.

For the special finite dimensional setting $V = \mathbb{C}^n$ and $Q = \mathbb{C}^m$, the linear operators 
above can be identified with corresponding matrices
\[  
  A \in \mathbb{C}^{n\times n},\quad
  B \in \mathbb{C}^{n\times m}, \quad
  B^*\in \mathbb{C}^{m\times n}, \quad
  C \in \mathbb{C}^{m\times n},
\]
where $B^*$ denotes the conjugate transpose of a matrix $B$. Then (\ref{M}) becomes a linear system of equations in standard matrix-vector notation.

A variety of interesting problems in science and engineering can be framed in this way, for example, mixed formulations of elliptic boundary value problems, the Stokes problem in fluid mechanics and the optimality system of PDE-constraint optimization problems
, see, e.g., \cite{lions:71}, \cite{brezzi:91}, \cite{benzi:05}, \cite{troeltzsch:10}, \cite{hinze:09}. Nonlinear versions of such problems are typically approximated by a sequence of systems of the form (\ref{M}) during a linearization process, see, e.g, \cite{troeltzsch:10}.
Discretized versions of these problems lead to finite dimensional linear systems of the form (\ref{M}) with large scale and sparse matrices $A$, $B$ and $C$, see, e.g, \cite{brezzi:91}.
The use of a framework based on complex rather than real Hilbert spaces is motivated by particular applications in spaces of time-periodic functions, for which the Fourier transform naturally introduces a representation in complex Hilbert spaces.
 
We will address two fundamental and related topics: stability estimates for (\ref{M}) in general Hilbert spaces and spectral properties of the matrix $\mathcal{M}$ in the finite dimensional setting for Hermitian matrices $\mathcal{M}$. 

Stability estimates for (\ref{M}) are estimates of the norm of the solution 
\[
  x = \begin{bmatrix} u \\ p \end{bmatrix}  \in X = V \times Q
\]
in terms of the norm of the right-hand side 
\[
  \begin{bmatrix} f \\ g \end{bmatrix}  \in V^* \times Q^* .
\]
The product space $V^* \times Q^*$ can be identified with $X^*$, the dual of $X$, in a canonical way. That makes $\mathcal{M}$ a bounded linear operator from $X$ to $X^*$ and a stability estimate can be written in form of an upper bound for the corresponding norm of $\mathcal{M}^{-1}$:
\begin{equation} \label{lower}
   \|\mathcal{M}^{-1}\|_{L(X^*,X)} \le \frac{1}{\underline{c}} ,
\end{equation}
where  $\|\cdot \|_H$ denotes the norm in a Hilbert space $H$.
Of course, such an estimate depends on the choice of the norm or better the inner product in $X$.
We will concentrate on inner products in $X$ of the particular form
\begin{equation} \label{innerprodX}
  \big( x, y \big)_X = ( u, v )_V + ( p, q )_Q
  \quad \text{for} \quad
  x = \begin{bmatrix} u \\ p \end{bmatrix}, \ 
  y = \begin{bmatrix} v \\ q \end{bmatrix}, \ 
\end{equation}
where $(\cdot,\cdot)_H$ denotes the inner product in a Hilbert space $H$.

Stability estimates are required for showing the well-posedness of (\ref{M}) and are also important for discretization error estimates. 
Rather sharp stability bounds $\underline{c}$ for the case $C = 0$ are already contained in the pioneering paper \cite{brezzi:74} and have been improved, e.g., in \cite{xu:03}. For $C = 0$ we will show how to further improve the estimates and present the best possible bounds in a general Hilbert space setting under the same assumptions as in \cite{brezzi:74} and \cite{xu:03}. For general linear operators $C$ we will formulate a simple but rather helpful criterion which leads to stability estimates for certain problems from optimal control.

Note that, contrary to (\ref{lower}), sharp upper bounds of the form
\begin{equation} \label{upper}
   \|\mathcal{M}\|_{L(X,X^*)} \le \overline{c}
\end{equation}
are easy to obtain and well-known. 

For $V = \mathbb{C}^n$ and $Q = \mathbb{C}^m$, inner products can be represented by Hermitian and positive definite matrices. In particular, (\ref{innerprodX}) corresponds to inner products of the following form
\[
  \big( x, y \big)_X = \langle P u,v \rangle + \langle R p,q \rangle = \langle \mathcal{P} x,y \rangle 
  \quad \text{with} \quad
  \mathcal{P} = \begin{bmatrix} P & 0 \\ 0 & R \end{bmatrix},
\]
where $\langle \cdot,\cdot \rangle$ denotes the Euclidean inner product in $\mathbb{C}^r$ and $P \in \mathbb{C}^{n \times n}$, $R \in \mathbb{C}^{m \times m}$, and, therefore, also $\mathcal{P} \in \mathbb{C}^{(n+m)\times (n+m)}$ are Hermitian and positive definite matrices.

In the Hermitian case $\mathcal{M}^* = \mathcal{M}$ the estimates (\ref{lower}) and (\ref{upper}) immediately lead to corresponding estimates 
\[
  \underline{c} \le |\mu| \le \overline{c},
\]
for the eigenvalues $\mu$ of the matrix 
$\widehat{\mathcal{M}} = \mathcal{P}^{-1} \mathcal{M}$.
Therefore, we have 
\[
  \mu \in [-\overline{c},-\underline{c}] \cup [\underline{c},\overline{c}].
\]
We will show that inclusions
 of this form with two intervals symmetrically arranged around 0 are appropriate for the case $C = A$ and $R = P$ due to symmetry properties of the spectrum of $\widehat{\mathcal{M}}$. 
For the case $C = 0$ and $A$ positive definite,
sharp estimates of the more general form
\begin{equation} \label{mu}
  \mu \in [\mu_1,\mu_2] \cup [\mu_3,\mu_4]
\end{equation}
with $\mu_1 \le \mu_2 < 0 < \mu_3 \le \mu_4$ were derived 
in \cite{rusten:92}.
For  positive semidefinite $C$ and indefinite $A$, 
rather sharp estimates were presented in \cite{gould:10}. For the special case $C = 0$ we will further improve these estimates and present the best possible bounds under similar assumptions as in \cite{gould:10}.

Preconditioned Krylov subspace methods are an important class of iterative methods for solving (\ref{M}) in the finite dimensional setting with large and sparse matrices $A$, $B$, and $C$;
we refer to the survey article \cite{benzi:05} and the references cited there for
these as well as other solution techniques for saddle point problems. In the Hermitian case $\mathcal{M}^* = \mathcal{M}$ a prominent representative of these methods is the preconditioned minimal residual method (MINRES), see \cite{paige:75}, whose convergence analysis relies
on inclusions of the form (\ref{mu}), see, e.g., \cite{paige:75}, \cite{greenbaum:97}.

The paper is organized as follows. 
Sect.~\ref{general} contains stability estimates of the form (\ref{lower}) in general Hilbert spaces for $C = 0$. 
In Sect.~\ref{Hermitian} we study the finite dimensional Hermitian case and derive sharp estimates of the form (\ref{mu}) for $C = 0$. 
Structural properties of the spectrum of $\widehat{\mathcal{M}}$ for $C = A$ and $R = P$ are studied in Sect.~\ref{symspectrum}.
In Sect.~\ref{applications} the general results of the preceding sections are applied to two model problems from optimal control, distributed time-periodic parabolic control and distributed time-periodic Stokes control. The paper finishes with concluding remarks and some technical details of a proof shifted to the appendix.

\section{The general case}
\label{general}

For the analysis we use a reformulation of (\ref{M}) as a variational problem: 
The linear operators $A$, $B$, and $C$ uniquely determine sesquilinear forms $a$, $b$, and $c$ on $V \times V$, $V \times Q$ and $Q \times Q$, respectively, given by
\begin{equation} \label{ABC}
  a(u,v) = \langle Au,v \rangle, \quad
  b(v,q) = \langle Bv,q \rangle = \langle B^*q,v \rangle, \quad
  c(p,q) = \langle Cp,q \rangle ,
\end{equation}
where, in general, $\langle \cdot,\cdot \rangle$ denotes the duality pairing in Hilbert spaces. (It reduces to the Euclidean inner product, as used in the introduction, in matrix-vector notation in the finite dimensional setting.)
Then (\ref{M}) becomes a mixed variational problem: 
For given $f \in V^*$ and $g \in Q^*$, find $u \in V$ and $p \in Q$ such that
\begin{equation} \label{abc}
\begin{aligned}
  a(u,v) + b(v,p) & = f(v) \quad \text{for all} \ v \in V, \\
  b(u,q) - c(p,q) & = g(q) \quad \text{for all} \ q \in Q,
\end{aligned}
\end{equation}
or, equivalently, a variational problem in the product space $X = V \times Q$: Find $x \in X$ such that
\begin{equation} \label{B}
  \mathcal{B}(x,y) =  \mathcal{F}(y) \quad \text{for all} \ y \in X
\end{equation}
with
\[
  \mathcal{B}(x,y) = \langle \mathcal{M}x,y \rangle = a(u,v) + b(v,p) + b(u,q) - c(p,q),
  \quad
  \mathcal{F}(y) = f(v) + g(q)
\]
for
\[
  x = \begin{bmatrix} u \\ p \end{bmatrix}
  \quad \text{and} \quad
  y = \begin{bmatrix} v \\ q \end{bmatrix} .
\]

From the famous Babu\v{s}ka-Brezzi theory, see \cite{babuska:71}, \cite{babuska:73}, \cite{brezzi:74}, applied to the case $C = 0$, it is well-known that the following three statements are equivalent:
\begin{enumerate}
\item
$\mathcal{M}$ is an isomorphism between $X$ and $X^*$
\item
$\|\mathcal{B}\| < \infty$ with
\[
   \|\mathcal{B}\| = \sup_{0 \neq x \in X} \sup_{0 \neq y \in X} 
                    \frac{| \mathcal{B}(x,y) |}{\|x\|_X \|y\|_X} 
\]
and
\[
  \inf_{\rule[0.6ex]{0ex}{1ex} 0 \neq x \in X} \sup_{0 \neq y \in X} 
           \frac{|\mathcal{B}(x,y)|}{\|x\|_X \|y\|_X}
   = \inf_{\rule[0.6ex]{0ex}{1ex} 0 \neq y \in X} \sup_{0 \neq x \in X} 
           \frac{|\mathcal{B}(x,y)|}{\|x\|_X \|y\|_X} \equiv \gamma > 0.
\]
\item
$\|a\| < \infty$, $\|b\| < \infty$ with
\[
  \|a\| = \sup_{0 \neq u \in V} \sup_{0 \neq v \in V} 
                    \frac{| a(u,v) |}{\|u\|_V \|v\|_V}, \quad
  \|b\| = \sup_{0 \neq v \in V} \sup_{0 \neq q \in Q} 
                    \frac{| b(v,q) |}{\|v\|_V \|q\|_Q}, 
\]
and
\[
  \inf_{\rule[0.6ex]{0ex}{1ex} 0 \neq u \in \ker B} \sup_{0 \neq v \in \ker B} 
           \frac{|a(u,v)|}{\|u\|_V \|v\|_V}
   = \inf_{\rule[0.6ex]{0ex}{1ex} 0 \neq v \in \ker B} \sup_{0 \neq u \in \ker B} 
           \frac{|a(u,v)|}{\|u\|_V \|v\|_V} \equiv \alpha > 0,
\]
\[
  \inf_{\rule[0.6ex]{0ex}{1ex} 0 \neq q \in Q} \sup_{0 \neq v \in V} 
           \frac{|b(v,q)|}{\|v\|_V \|q\|_Q} \equiv \beta  > 0.
\]
\end{enumerate}
It is also well-known that the quantities $\|\mathcal{B}\|$ and $\gamma$ can be expressed in terms of norms of the associated operator $\mathcal{M}$:
\begin{equation} \label{gammaMinv}
  \|\mathcal{B}\| = \|\mathcal{M}\|_{L(X,X^*)}, \quad
  \frac{1}{\gamma} = \|\mathcal{M}^{-1}\|_{L(X^*,X)}.
\end{equation}
With these notations and relations, (\ref{lower}) and (\ref{upper}) take the following form
\[
  0 <  \underline{c} \le \gamma  \quad \text{and} \quad \|\mathcal{B}\| \le \overline{c} < \infty.
\]
Our aim is thus to derive lower bounds of $\gamma$ and upper bounds for $\|\mathcal{B}\|$. One possible approach to obtain these bounds is based on available lower bounds for $\alpha$, $\beta$ and upper bounds for $\|a\|$, $\|b\|$. These quantities can also be expressed in terms of norms of the associated operators $A$ and $B$.
We follow here the presentation in \cite{arnold:81} and introduce a space decomposition on $V$:
\[
  V = V_0 + V_1 \quad \text{with} \quad V_0 = \ker B \quad \text{and} \quad V_1 = V_0^\perp ,
\]
and define the operators $A_{ij} \in L(V_j,V_i^*)$ for all $i,j \in \{0,1\}$ by
\[
  \langle A_{ij} u,v \rangle = a(u,v) \quad \text{for all} \quad u \in V_j,\ v \in V_i
\]
and $B_1 \in L(V_1,Q^*)$ by
\[
  \langle B_1 v,q \rangle = b(v,q) \quad \text{for all} \quad v \in V_1, \ q \in Q.
\]
Then it is easy to see that $A_{00}$ and $B_1$ are isomorphisms and we have:
\begin{equation} \label{const1}
  \|a\| = \|A\|_{L(V,V^*)}, \quad
  \|b\| = \|B\|_{L(V,Q^*)} = \|B_1\|_{L(V_1,Q^*)} = \|B_1^*\|_{L(Q,V_1^*)},
\end{equation}
and
\begin{equation} \label{const2}
  \frac{1}{\alpha} = \|A_{00}^{-1}\|_{L(V_0^*,V_0)}, \quad
  \frac{1}{\beta} = \|B_1^{-1}\|_{L(Q^*,V_1)} = \|(B_1^*)^{-1}\|_{L(V_1^*,Q)}.
\end{equation}

With these operators we can rewrite the original problem in a 3-by-3 block form, as suggested in \cite{arnold:81}:
\begin{equation} \label{3by3}
  \begin{bmatrix}
    A_{00} & A_{01} & 0 \\
    A_{10} & A_{11} & B_1^* \\
    0 & B_1 & 0 
  \end{bmatrix}
  \begin{bmatrix} u_0 \\ u_1 \\ p \end{bmatrix}
   =
  \begin{bmatrix} f_0 \\ f_1 \\ g \end{bmatrix} ,
\end{equation}
with $u = u_0 + u_1$, $u_i \in V_i$ and $f_i \in V_i^*$ are given by $\langle f_i,v \rangle = \langle f,v \rangle$, $v \in V_i$ for $i \in \{0,1\}$. As observed in \cite{arnold:81} the inverse of the 3-by-3 block operator in (\ref{3by3}) is given by
\begin{align*}
   \begin{bmatrix}
      A_{00}^{-1} & 0 & - A_{00}^{-1}A_{01}  B_1^{-1} \\
      0 & 0 & B_1^{-1} \\
      - (B_1^*)^{-1} A_{10} A_{00}^{-1} & (B_1^*)^{-1} & - (B_1^*)^{-1} \left[ A_{11} -  A_{10} A_{00}^{-1} A_{01} \right]  B_1^{-1} 
    \end{bmatrix} .
\end{align*}
Using this representation and (\ref{const1}), (\ref{const2}) it follows immediately for (\ref{3by3}) that
\begin{align*}
  \begin{bmatrix} \|u_0\|_V \\[1ex] \|u_1\|_V \\[1ex] \|p\|_Q \end{bmatrix}
  & \le 
  \begin{bmatrix}
      \frac{1}{\alpha} & 0 & \frac{1}{\beta} \, \left\|A_{00}^{-1}A_{01} \right\| \\
        0 & 0 & \frac{1}{\beta} \\
      \frac{1}{\beta} \, \left\|A_{10} A_{00}^{-1} \right\| & \frac{1}{\beta} & \frac{1}{\beta^2} \, 
      \left\| A_{11} -  A_{10} A_{00}^{-1} A_{01} \right\| 
  \end{bmatrix}
  \begin{bmatrix} \|f_0\|_{V_0^*} \\[1ex] \|f_1\|_{V_1^*} \\[1ex] \|g\|_{Q^*} \end{bmatrix}.
\end{align*}
Here, we dropped indices of the operator norms for simplicity. It is clear from the context which operator norm is meant.
Obviously, we have $\|A_{ij}\| \le \|A\| = \|a\|$ for $i \in \{0,1 \}$. Therefore,
\begin{equation} \label{old1}
  \left\|A_{00}^{-1} A_{01} \right\| \le \frac{\|a\|}{\alpha}, \quad
  \left\|A_{10} A_{00}^{-1} \right\| \le \frac{\|a\|}{\alpha},
\end{equation}
and
\begin{equation} \label{old2}
  \left\| A_{11} -  A_{10} A_{00}^{-1} A_{01} \right\| \le \|a\| + \frac{\|a\|^2}{\alpha},
\end{equation}
which imply
\begin{equation} \label{est1}
  \begin{bmatrix} \|u_0\|_V \\[1ex] \|u_1\|_V \\[1ex] \|p\|_Q \end{bmatrix}
    \le 
  \begin{bmatrix}
      \frac{1}{\alpha} & 0 & \frac{\|a\|}{\alpha \beta} \\
        0 & 0 & \frac{1}{\beta} \\
      \frac{\|a\|}{\alpha \beta}  & \frac{1}{\beta} & \frac{\|a\|}{\beta^2} \, 
      \left( 1 + \frac{\|a\|}{\alpha} \right) 
  \end{bmatrix}
  \begin{bmatrix} \|f_0\|_{V_0^*} \\[1ex] \|f_1\|_{V_1^*} \\[1ex] \|g\|_{Q^*} \end{bmatrix} .
\end{equation}
Using
\[
  \|u\|_V \le \|u_0\|_V + \|u_1\|_V 
  \quad \text{and} \quad
  \|f_0\|_{V_0^*} \le \|f\|_{V^*}, \quad
  \|f_1\|_{V_1^*} \le \|f\|_{V^*} ,
\]
it easily follows that
\begin{align*}
  \begin{bmatrix} \|u\|_V \\[1ex] \|p\|_Q \end{bmatrix}
  & \le 
  \begin{bmatrix}
      \frac{1}{\alpha} & \frac{1}{\beta} \, 
      \left( 1 + \frac{\|a\|}{\alpha} \right) \\
      \frac{1}{\beta} \, 
      \left( 1 + \frac{\|a\|}{\alpha} \right) & \frac{\|a\|}{\beta^2} \, 
      \left( 1 + \frac{\|a\|}{\alpha} \right) 
  \end{bmatrix}
  \begin{bmatrix} \|f\|_{V^*} \\[1ex] \|g\|_{Q^*} \end{bmatrix}.
\end{align*}
These are exactly the estimates that can be already found in \cite{brezzi:74} and later in \cite{arnold:81} with essentially the same line of arguments. Since
\[
  \begin{bmatrix} u \\ p \end{bmatrix} = \mathcal{M}^{-1} \begin{bmatrix} f \\ g \end{bmatrix}
\]
and $f \in V^*$, $g \in Q^*$ can be chosen arbitrarily, this estimate shows that 
\begin{equation} \label{D1}
  \|\mathcal{M}^{-1}\|_{L(X^*,X)} \le \rho(D_1) \quad \text{with} \quad
  D_1 = \begin{bmatrix}
      \frac{1}{\alpha} & \frac{1}{\beta} \, 
      \left( 1 + \frac{\|a\|}{\alpha} \right) \\
      \frac{1}{\beta} \, 
      \left( 1 + \frac{\|a\|}{\alpha} \right) & \frac{\|a\|}{\beta^2} \, 
      \left( 1 + \frac{\|a\|}{\alpha} \right) 
  \end{bmatrix} ,
\end{equation}
where $\rho(M)$ denotes the spectral radius of a matrix $M$.
From (\ref{gammaMinv}) we learn that $1/\rho(D_1)$ is a lower bound for $\gamma$. 

With a slight variation of the arguments we obtain a better lower bound for $\gamma$. The essential step is an improvement of the estimates in (\ref{old1}), (\ref{old2}).

\begin{lemma} \label{lem1}
We have
\[
  \| A_{10} A_{00}^{-1} \| \le \sqrt{\frac{\|a\|^2}{\alpha^2} - 1}, \quad
  \| A_{00}^{-1} A_{01} \| \le \sqrt{\frac{\|a\|^2}{\alpha^2} - 1}
\]
and
\[
  \left\|  A_{11} -  A_{10} A_{00}^{-1} A_{01} \right\|  \le \frac{\|a\|^2}{\alpha}.
\]
\end{lemma}
\begin{proof}
By using the simple estimate
\[
  \left\| \begin{bmatrix} A_{00} \\ A_{10} \end{bmatrix} \right\| 
    \le \left \|\begin{bmatrix} A_{00} & A_{01} \\ A_{10} & A_{11} \end{bmatrix} \right\|
     = \|A\| = \|a\| ,
\]
it follows that:
\[
  1 + \|A_{10} A_{00}^{-1}\|^2 
    = \left\| \begin{bmatrix} I \\ A_{10}A_{00}^{-1} \end{bmatrix} \right\|^2
    = \left\| \begin{bmatrix} A_{00} \\ A_{10} \end{bmatrix} A_{00}^{-1} \right\|^2
   \le \frac{\|a\|^2}{\alpha^2},
\]
from which the first inequality follows immediately. The second inequality can be shown analogously.
For the third estimate we have
\begin{align*}
  \| A_{10} A_{00}^{-1} A_{01} - A_{11}\| 
  & = \left\| \begin{bmatrix} A_{10} A_{00}^{-1} & - I  \end{bmatrix}  \begin{bmatrix} A_{01} \\ A_{11} \end{bmatrix} \right\|
  \le \left\| \begin{bmatrix} A_{10} A_{00}^{-1} & - I  \end{bmatrix} \right\| \left\|  \begin{bmatrix} A_{01} \\ A_{11} \end{bmatrix} \right\| \\
  & = \sqrt{1 + \|A_{10} A_{00}^{-1}\|^2} 
       \left\| \begin{bmatrix} A_{01} \\ A_{11} \end{bmatrix} \right\|
   \le \frac{\|a\|}{\alpha} \, \|a\| = \frac{\|a\|^2}{\alpha} ,
\end{align*} 
which completes the proof.
\end{proof}

Using the estimates of the previous lemma we obtain the following result.

\begin{theorem} \label{mr}
With the notation introduced above, the following holds:
\begin{enumerate}
\item
Let $\gamma_\text{opt}(\alpha,\beta,\|a\|)$ be the smallest positive root of the cubic equation
\begin{equation} \label{cubic}
  \mu^3 - (\|a\|^2 + \beta^2) \mu + \alpha \, \beta^2  = 0.
\end{equation}
Then
\item
We have
\begin{equation} \label{D2}
  \begin{bmatrix} \|u\|_V \\[1ex] \|p\|_Q \end{bmatrix}
   \le \begin{bmatrix}
      \frac{1}{\alpha} & \frac{1}{\beta} \, 
      \frac{\|a\|}{\alpha} \\[1ex]
      \frac{1}{\beta} \, 
      \frac{\|a\|}{\alpha} & \frac{\|a\|}{\beta^2} \, 
      \frac{\|a\|}{\alpha}
  \end{bmatrix} \, 
  \begin{bmatrix} \|f\|_{V^*} \\[1ex] \|g\|_{Q^*} \end{bmatrix}
\end{equation}
and $\gamma \ge \alpha/(1 + \kappa^2)$ with $\kappa = \|a\|/\beta$.
\end{enumerate}
\end{theorem}

\begin{proof}
From (\ref{est1}) and Lemma \ref{lem1} it follows that
\begin{equation} \label{est3}
  \begin{bmatrix} \|u_0\|_V \\[1ex] \|u_1\|_V \\[1ex] \|p\|_Q \end{bmatrix}
    \le 
   E \, 
  \begin{bmatrix} \|f_0\|_{V_0^*} \\[1ex] \|f_1\|_{V_1^*} \\[1ex] \|g\|_{Q^*} \end{bmatrix},
  \  \text{with} \ 
  E = \begin{bmatrix}
      \frac{1}{\alpha} & 0 & \frac{1}{\beta} \, \sqrt{\frac{\|a\|^2}{\alpha^2} - 1} \\
        0 & 0 & \frac{1}{\beta} \\
      \frac{1}{\beta} \, \sqrt{\frac{\|a\|^2}{\alpha^2} - 1} & \frac{1}{\beta} & \frac{1}{\beta^2} \,  \frac{\|a\|^2}{ \alpha},
   \end{bmatrix} .
\end{equation}
Therefore,
\[
  \left\| \begin{bmatrix} u \\ p \end{bmatrix} \right\|_X \le \rho(E) \left\| \begin{bmatrix} f \\ g \end{bmatrix} \right\|_{X^*},
\]
which implies that $1/\rho(E)$ is a lower bound for $\gamma$. 
Since $E$ is a nonnegative matrix, $\rho(E)$ is an eigenvalue of $E$, 
see \cite[page 26]{berman:79}. Therefore, 
$\rho(E)$ is the largest positive root of the characteristic polynomial of $E$, which directly leads to the cubic equation (\ref{cubic}) for $1/\rho(E) = \gamma_\text{opt}(\alpha,\beta,\|a\|)$. 
Using
\[
  \|u\|_V^2 = \|u_0\|_V^2 + \|u_1\|_V^2 ,
  \quad  \|f\|_{V^*}^2 = \|f_0\|_{V_0^*}^2 + \|f_1\|_{V_1^*}^2 ,
\]
and Cauchy's inequality, it easily follows from (\ref{est3}) that
\[
  \begin{bmatrix} \|u\|_V \\[1ex] \|p\|_Q \end{bmatrix}
   \le D_2 \, 
  \begin{bmatrix} \|f\|_{V^*} \\[1ex] \|g\|_{Q^*} \end{bmatrix} ,
  \quad \text{with} \quad
  D_2 = \begin{bmatrix}
      \frac{1}{\alpha} & \frac{1}{\beta} \, 
      \frac{\|a\|}{\alpha} \\[1ex]
      \frac{1}{\beta} \, 
      \frac{\|a\|}{\alpha} & \frac{\|a\|}{\beta^2} \, 
      \frac{\|a\|}{\alpha}
  \end{bmatrix} .
\]
As above, this implies that $1/\rho(D_2)$ is a lower bound for $\gamma$. 
Observe that $D_2$ is a rank-one matrix
\[
  D_2
   = \frac{1}{\alpha} \, \begin{bmatrix}  1 \\ \kappa \end{bmatrix}
      \begin{bmatrix}  1 & \kappa \end{bmatrix} ,
  \quad \text{with} \quad \kappa = \frac{\|a\|}{\beta},
\]
whose spectrum is given by $\{ 0, (1+\kappa^2)/\alpha \}$. 
Therefore, $1/\rho(D_2) = \alpha/( 1 + \kappa^2)$, 
which completes the proof. 
\end{proof}

We would like to remark
that the estimate $\gamma \ge \gamma_3(\alpha,\beta,\|a\|)$ is sharp. 
Indeed, for the matrix $\mathcal{M}$ in (\ref{M})
with
\begin{eqnarray}\label{eqn:example}
   A = \begin{bmatrix} \alpha & - \sqrt{\|a\|^2 - \alpha^2} \\
      - \sqrt{\|a\|^2 - \alpha^2} & - \alpha
       \end{bmatrix} , \quad
   B = \begin{bmatrix} 0 & \beta \end{bmatrix} 
  \quad \text{and} \quad C=0,
\end{eqnarray}
we obtain
\begin{align*}
  \mathcal{M}^{-1}
   =
   \begin{bmatrix}
      \frac{1}{\alpha} & 0 & \frac{1}{\beta} \, \sqrt{\frac{\|a\|^2}{\alpha^2} - 1} \\
        0 & 0 & \frac{1}{\beta} \\
      \frac{1}{\beta} \, \sqrt{\frac{\|a\|^2}{\alpha^2} - 1} & \frac{1}{\beta} & \frac{1}{\beta^2} \,  \frac{\|a\|^2}{ \alpha} 
   \end{bmatrix} .
\end{align*}
The conditions of the Babu\v{s}ka-Brezzi theory are satisfied with constants $\|a\|$, $\alpha$, and $\beta$ for $\mathcal{P} = I$, the identity matrix. Since $\mathcal{M}^{-1}$ coincides with the matrix $E$, it immediately follows that
\[
  \gamma = \gamma_\text{opt}(\alpha,\beta,\|a\|).
\]

We would also like to emphasize that both bounds of 
Theorem~\ref{mr} are sharper than the classical bound $1/\rho(D_1)$. Indeed,
if we apply the estimates for $\gamma$ from Theorem~\ref{mr} to the matrices
in (\ref{eqn:example}), we 
obtain $\gamma = \gamma_\text{opt}(\alpha,\|a\|,\|b\|) \ge \alpha/(1 + \kappa^2)$. 
Actually, we have a strict inequality
\[
  \gamma_\text{opt}(\alpha,\|a\|,\|b\|) > \frac{\alpha}{1 + \kappa^2},
\]
since the cubic polynomial in (\ref{cubic}) is strictly positive on $[0,\alpha/(1 + \kappa^2)]$. 
Moreover, for
\[
  D(\xi) = \frac{1}{\alpha} \begin{bmatrix} 1 & \kappa \xi \\ \kappa \xi & \kappa^2 \xi \end{bmatrix} ,
\]
we have
\[
  D_2 = D(1) \quad \text{and} \quad 
  D_1 = D(\xi_1) \quad \text{with} \quad \xi_1 = \frac{\alpha + \|a\|}{\|a\|} > 1 .
\]
It is easy to check that $\rho(D(\xi))$ is a strictly increasing function of $\xi$. Therefore, $\rho(D_1) > \rho(D_2)$, or, equivalently,
\[
  \frac{\alpha}{1 + \kappa^2} > \frac{1}{\rho(D_1)}.
\]

Comparing the results in Theorem~\ref{mr}
with similar bounds in the literature, we notice that
the
bound presented in \cite{xu:03} is of the same form, 
say $\tilde{\gamma}(\alpha,\beta,\|a\|)$, and one can show that
\[
  \frac{\alpha}{1 + \kappa^2}
   > \tilde{\gamma}(\alpha,\beta,\|a\|)
   > \frac{1}{\rho(D_1)}.
\]

For completeness we finally 
include a trivial upper bound for $\|\mathcal{B}\|$. Since we have
\[
  \begin{bmatrix} \|f\|_{V^*} \\ \|g\|_{Q^*} \end{bmatrix}
   \le 
  \begin{bmatrix}
      \|a\| & \|b\|  \\
      \|b\| & 0
  \end{bmatrix}
  \begin{bmatrix} \|u\|_V \\ \|p\|_Q \end{bmatrix} ,
\]
it follows that
\begin{equation} \label{upperbound}   
  \|\mathcal{B}\| \le \rho\left(\begin{bmatrix}
      \|a\| & \|b\|  \\
      \|b\| & 0
  \end{bmatrix} \right)
   = \frac{1}{2} \left( \|a\| + \sqrt{ \|a\|^2 + 4 \, \|b\|^2} \right).
\end{equation}
It is well known that this estimate is sharp;
see, e.g., \cite{rusten:92}.

\section{The finite dimensional Hermitian case}
\label{Hermitian}

In this section we consider the finite dimensional case $V = \mathbb{C}^n$ and $Q = \mathbb{C}^m$.
Inner products in these spaces can be represented by Hermitian and positive definite matrices $P \in \mathbb{C}^{n \times n}$ and $R \in \mathbb{C}^{m \times m}$:
\[
  (u,v)_V = \langle u,v \rangle_P, \quad (p,q)_Q = \langle p,q \rangle_R,
\]
where here and in the sequel the following standard notations are used:
For a Hermitian and positive definite matrix $M \in \mathbb{C}^{r \times r}$, the 
associated inner product is given by $\langle z,w \rangle_M = \langle M z,w \rangle$ for $z, w \in \mathbb{C}^r$. Both the vector norm and the matrix norm associated with the inner product $\langle \cdot , \cdot \rangle_M$ are denoted by $\|\cdot \|_M$.

With the notation introduced above,
the inner product (\ref{innerprodX}) in $X = V \times Q = \mathbb{C}^{n + m}$ is given by 
\begin{equation} \label{xnorm}
  (x,y)_X = \langle x,y \rangle_\mathcal{P}
  \quad \text{with} \quad
  \mathcal{P} = \begin{bmatrix} P & 0 \\ 0 & R \end{bmatrix} \in \mathbb{C}^{(n+m)\times (n+m)}.
\end{equation}

The discussion in this section is restricted to the Hermitian case:
\[
  \mathcal{M}^* = \mathcal{M}. 
\]
For the inf-sup constant $\gamma$ and the norm $\|\mathcal{B}\|$ of the bilinear form $\mathcal{B}$ associated with $\mathcal{M}$ as introduced in Sect.~\ref{general}, it is well-known that
\[
  \gamma = |\mu_{\text{min}}|
  \quad \text{and} \quad
  \|\mathcal{B}\| = |\mu_{\text{max}}|,
\]
where $\mu_{\text{min}}$ and $\mu_{\text{min}}$ are the eigenvalues of the matrix $\widehat{\mathcal{M}} = \mathcal{P}^{-1} \mathcal{M}$ with minimal and maximal modulus, respectively. Or, equivalently, for the eigenvalues $\mu$ of the generalized eigenvalue problem
\begin{equation} \label{MP}
  \mathcal{M} x = \mu \, \mathcal{P} x ,
\end{equation}
we have
\begin{equation} \label{symbound}
  \mu \in \left[ - \|\mathcal{B}\|,\, -\gamma \, \right]\,\, \cup\,\, \left[\,\gamma ,\, \|\mathcal{B}\| \,\right] .
\end{equation}

For the rest of this section we concentrate on matrices of the form
\begin{eqnarray}\label{eqn:M}
  \mathcal{M} = \begin{bmatrix} A & B^* \\ B & 0 \end{bmatrix} \in \mathbb{C}^{(n+m)\times (n+m)}
\end{eqnarray}
with $A^* = A \in \mathbb{C}^{n \times n}$ satisfying
\[
  \langle A v, v \rangle > 0 \quad \text{for all} \ 0 \neq v \in \ker B,
\]
and $B \in \mathbb{C}^{m \times n}$ of full rank $m \le n$. 
Under these conditions the matrix $\mathcal{M}$ is Hermitian and non-singular.
This setting is typical for optimality systems (Karush-Kuhn-Tucker conditions) of equality-constrained optimization problems.
Let $P_i \in L(V_i,V_i^*)$ be given by
\[
  \langle P_i u,v \rangle = \langle P u,v \rangle \quad \text{for all} \quad u, v \in V_i,
    \ i \in \{0,1\}.
\]
Under the conditions above
we have the following well-known alternative representation of the inf-sup constant $\alpha$:
\begin{equation} \label{alpha}
  \alpha = \inf_{0 \neq v_0 \in V_0} \frac{\langle A_{00} v_0, v_0 \rangle }{\langle P_0  v_0 ,v_0 \rangle}.
\end{equation}
Instead of $\|a\|$, the norm of the bilinear form $a$ associated with
 $A$ in $V$, we assume more detailed information on $A$ in terms of 
the extreme eigenvalues of $P^{-1} A$, which can be written as
\begin{equation} \label{lambda}
  \lambda_\text{min}^A = \inf_{0 \neq v \in V} \frac{\langle A v, v \rangle }{\langle P  v ,v \rangle}, \quad
  \lambda_\text{max}^A = \sup_{0 \neq v \in V} \frac{\langle A v, v \rangle }{\langle P  v ,v \rangle} .
\end{equation}
Of course, we have $\lambda_\text{max}^A \ge \alpha > 0$ and 
$\|a\|= \max \{|\lambda_\text{min}^A|, \lambda_\text{max}^A \}$.

The following representations for the inf-sup constant $\beta$ and the norm $\| b\|$ of the bilinear form $b$ associated with $B$ in $Q$ hold:
\begin{equation} \label{beta}
  \beta^2 
   = \inf_{0 \neq q \in Q} \frac{\langle B P^{-1} B^* q,q \rangle }{\langle R q,q \rangle}
   = \inf_{0 \neq v_1 \in V_1} \frac{\langle B_1^* R^{-1} B_1 v_1,v_1 \rangle }{\langle P_1 v_1,v_1 \rangle} ,
\end{equation}
and
\begin{equation} \label{normb}
  \|b\|^2 = \sup_{0 \neq q \in Q} \frac{\langle B P^{-1} B^* q,q \rangle }{\langle R q,q \rangle}.
\end{equation}

Since $\mathcal{M}$ is Hermitian and indefinite we have an inclusion of the form
\[
  \mu \in [\mu_1,\mu_2] \cup  [\mu_3,\mu_4]
\]
with $\mu_1 \le \mu_2 < 0 < \mu_3 \le \mu_4$ for the eigenvalues $\mu$ of the generalized eigenvalue problem (\ref{MP}). 

Sharp bounds $\mu_1$, $\mu_2$, and $\mu_4$ in the case $\lambda_\text{min}^A > 0$ can be found in \cite{rusten:92}. In \cite{gould:10} these bounds were extended to the case $\lambda_\text{min}^A \le 0$ and read, in general,
\begin{equation} \label{mu124}
\begin{aligned}
  \mu_1
    & = \frac{1}{2} 
      \left( \lambda_\text{min}^A - \sqrt{(\lambda_\text{min}^A)^2 + 4  \|b\|^2 } \right), \\
  \mu_2 
    & = \frac{1}{2} 
      \left( \lambda_\text{max}^A - \sqrt{(\lambda_\text{max}^A)^2 + 4  \beta^2 } \right), \\
  \mu_4
    & = \frac{1}{2} 
      \left( \lambda_\text{max}^A + \sqrt{(\lambda_\text{max}^A)^2 + 4  \|b\|^2 } \right).
\end{aligned}
\end{equation}
It remains to discuss lower bounds for the positive eigenvalues. 
For the case $\lambda_\text{min}^A > 0$ a simple bound is known (see \cite{rusten:92}):
\[
  \mu_3 = \lambda_\text{min}^A .
\]
This bound is sharp for problems with $\alpha = \lambda_\text{min}^A$.
Rather sharp bounds with no restriction on $\lambda_\text{min}^A$ can be found in \cite{gould:10}. The best bound presented in \cite{gould:10} is the smallest root of a cubic equation whose coefficients are given in terms of $\alpha$, $\lambda_\text{min}^A$, $\lambda_\text{max}^A$, and $\beta$. 
For the case $\lambda_\text{min}^A \le 0$ we will derive now a sharp bound that involves a similar cubic equation.

\begin{theorem} \label{thmsym}
Assume that $\lambda_\text{min}^A \le 0$ and
let $\mu$ be any positive eigenvalue of $\mathcal{P}^{-1} \mathcal{M}$.
\begin{enumerate}
\item
Let $\gamma_\text{opt}(\alpha,\beta,\lambda_\text{min}^A,\lambda_\text{max}^A)$ be the smallest positive root of the cubic equation
\begin{equation} \label{cubic_sym}
  \mu^3 - (\lambda^A_{\text{min}} + \lambda^A_{\text{max}}) \, \mu^2 + (\lambda^A_{\text{min}}\lambda^A_{\text{min}} - \beta^2) \, \mu + \alpha \, \beta^2 = 0 .
\end{equation}
Then we have
$ \mu \ge \gamma_\text{opt}(\alpha,\beta,\lambda_\text{min}^A,\lambda_\text{max}^A)$ .
\item
We have 
\[
  \mu \ge
  \begin{cases}
     \displaystyle \frac{\alpha \beta^2}{ - \lambda_\text{min}^A \lambda_\text{max}^A + \beta^2}
      \qquad \qquad \quad \text{if} \quad \lambda_\text{min}^A + \lambda_\text{max}^A \le 0 , \\[3ex]
     \displaystyle
      \frac{\lambda_\text{min}^A \lambda_\text{max}^A - \beta^2}{2(\lambda_\text{min}^A + \lambda_\text{max}^A)} + 
      \sqrt{\left(\frac{\lambda_\text{min}^A \lambda_\text{max}^A - \beta^2}{2(\lambda_\text{min}^A + \lambda_\text{max}^A)}\right)^2 + \frac{\alpha \beta^2}{\lambda_\text{min}^A + \lambda_\text{max}^A}} \\[3ex]
      \qquad \qquad \qquad \qquad \qquad \qquad \text{otherwise} .
  \end{cases}
\]
\end{enumerate}
\end{theorem}

\begin{proof}
Let $\mu > 0$ be an eigenvalue of the eigenvalue problem (\ref{MP}) with eigenvector\linebreak $x = [u_0,u_1,p]^T \neq 0$. Using the framework introduced in \cite{arnold:81}, see Sect.~\ref{general}, this eigenvalue problem reads
\begin{align*}
  \begin{bmatrix}
      A_{00} & A_{01} & 0 \\
      A_{10} & A_{11} & B_1^* \\
      0      & B_1    & 0
  \end{bmatrix} 
  \begin{bmatrix} u_0 \\ u_1 \\ p \end{bmatrix}
   = 
  \mu \, 
  \begin{bmatrix} P_0 & 0 & 0 \\ 0 & P_1 & 0 \\ 0 & 0 & R \end{bmatrix}
  \begin{bmatrix} u_0 \\ u_1 \\ p \end{bmatrix} , 
\end{align*}
i.e.,
\begin{align*}
      A_{00} u_0 + A_{01} u_1           & \ = \mu \, P_0 u_0,   \\
      A_{10} u_0 + A_{11} u_1 + B_1^* p & \ = \mu \, P_1 u_1,  \\
      B_1 u_1                           & \ = \mu \, R p .
\end{align*}
First we consider the case that $A_{00} - \mu \, P_0$ is non-singular. 
Then we obtain from the first and the third equations
\[
   u_0 = - (A_{00} - \mu \, P_0)^{-1} A_{01} u_1 
   \quad \text{and} \quad
   p = \frac{1}{\mu} R^{-1} B_1 u_1,
\]
which immediately implies that $u_1 \neq 0$. Using these relations for eliminating $u_0$ and $p$ from the second equation and taking the inner product with $\mu \, u_1$ we obtain
\begin{equation} \label{nonlinear}
\begin{aligned}
   \mu^2 \, \langle P_1 u_1, u_1 \rangle
      + \mu \, \langle A_{10} (A_{00} - \mu \, P_0)^{-1} A_{01} u_1, u_1 \rangle \qquad &
     \\
   - \mu \, \langle A_{11} u_1,u_1 \rangle - \langle B_1^* R^{-1} B_1 u_1, u_1 \rangle & = 0.
\end{aligned}
\end{equation}
We have
$\langle A_{00} v_0,v_0 \rangle \ge \alpha \, \langle P_0 v_0,v_0 \rangle$ for all $v_0 \in V_0$;
see (\ref{alpha}). Then, for $0 < \mu < \alpha$, it is easy to check that 
\[
   \langle (A_{00} - \mu \, P_0)^{-1} v_0,v_0 \rangle \le \frac{\alpha}{\alpha - \mu} \, \langle A_{00}^{-1} v_0,v_0 \rangle
  \quad \text{for all} \quad v_0 \in V_0 .
\]
With $v_0 = A_{01} u_1$, an upper bound for the second term on the left-hand side in (\ref{nonlinear}) follows:
\begin{equation} \label{term2}
   \mu \, \langle A_{10} (A_{00} - \mu \, P_0)^{-1} A_{01} u_1, u_1 \rangle 
    \le \frac{\mu \alpha}{\alpha - \mu} \, \langle A_{10} A_{00}^{-1} A_{01} u_1, u_1 \rangle.
\end{equation}
For the last term on the left-hand side in (\ref{nonlinear}) we have
\begin{equation} \label{term4}
  \langle B_1^* R^{-1} B_1 u_1, u_1 \rangle \ge \beta^2 \, \langle P_1 u_1,u_1 \rangle ,
\end{equation}
see (\ref{beta}). Using (\ref{term4}) and (\ref{term2}) we obtain from (\ref{nonlinear}):
\[
   \mu^2 \, \langle P_1 u_1, u_1 \rangle
      + \frac{\mu \alpha}{\alpha - \mu} \, \langle A_{00}^{-1} A_{01} u_1,A_{01} u_1 \rangle
   - \mu \, \langle A_{11} u_1,u_1 \rangle - \beta^2 \, \langle P_1 u_1,u_1 \rangle \ge 0.
\]
After dividing by $\langle P_1 u_1, u_1 \rangle$ it follows that
\begin{equation} \label{nonlinear2}
   \mu^2 
      + \mu \left[ \frac{\alpha}{\alpha - \mu} \, r_1 - r_2 \right] - \beta^2 \ge 0 ,
\end{equation}
with the Rayleigh quotients
\[
   r_1 = \frac{\langle A_{10} A_{00}^{-1} A_{01} u_1, u_1 \rangle}{\langle P_1 u_1, u_1 \rangle}
   \quad \text{and} \quad
   r_2 = \frac{\langle A_{11} u_1,u_1 \rangle}{\langle P_1 u_1, u_1 \rangle} .
\]
One can show that (see the appendix for the technical details):
\begin{equation} \label{maxval}
   \frac{\alpha}{\alpha - \mu} \, r_1 - r_2 
    \le \frac{(\lambda_\text{max}^A + \lambda_\text{min}^A - \alpha) \, \mu  - \lambda_\text{max}^A \lambda_\text{min}^A} {\alpha - \mu} . 
\end{equation}
Then it follows from (\ref{nonlinear2}) that
\[
    \mu^2 
        + \mu \, \left[ \frac{(\lambda_\text{max}^A + \lambda_\text{min}^A - \alpha) \, \mu - \lambda_\text{max}^A \lambda_\text{min}^A} {\alpha - \mu} \right] - \beta^2 \ge 0 ,
\]
i.e.
\[
    q(\mu) = \mu^3
        - (\lambda_\text{max}^A + \lambda_\text{min}^A) \, \mu^2 + ( \lambda_\text{max}^A \lambda_\text{min}^A- \beta^2) \, \mu 
        + \alpha \beta^2 \le 0. 
\]
Therefore, $\mu$ cannot lie in the interval between 0 and the first positive root, denoted by\linebreak $\gamma_\text{opt}(\alpha,\beta,\lambda_\text{min}^A,\lambda_\text{max}^A)$, of the cubic polynomial $q(\mu)$, because there the cubic polynomial is strictly positive.

If $A_{00} - \mu \, P$ is singular, then $\mu \ge \alpha$ because of (\ref{alpha}). Since $q(\alpha) \le 0$, it follows also in this case that $\mu \ge \gamma_\text{opt}(\alpha,\beta,\lambda_\text{min}^A,\lambda_\text{max}^A)$. That completes the proof of the first part.

The proof of the second part follows the same line of arguments as presented in the proof of Proposition 2.2 in \cite{gould:10} and is omitted here. 
\end{proof}

\vskip 0.1in
We remark that
the estimate $\mu \ge \gamma_\text{opt}(\alpha,\beta,\lambda_\text{min}^A,\lambda_\text{max}^A)$ is sharp: 
For the matrix $\mathcal{M}$ in (\ref{eqn:M})
with
\[
   A = \begin{bmatrix} 
          \alpha & - \sqrt{(\lambda_\text{max}^A - \alpha)(\alpha - \lambda_\text{min}^A)} \\
          - \sqrt{(\lambda_\text{max}^A - \alpha)(\alpha - \lambda_\text{min}^A)} & \lambda_\text{max}^A + \lambda_\text{min}^A - \alpha
       \end{bmatrix}, \quad
   B = \begin{bmatrix} 0 & \beta \end{bmatrix}
\]
and $\mathcal{P} = I$, we obtain the following characteristic polynomial
\[
    \mu^3
        - (\lambda_\text{max}^A + \lambda_\text{min}^A) \, \mu^2 + ( \lambda_\text{max}^A \lambda_\text{min}^A- \beta^2) \, \mu 
        + \alpha \beta^2 ,
\]
which coincides with the cubic polynomial in (\ref{cubic}).
Therefore, the smallest positive eigenvalue of $\mathcal{M}$ is equal to $\gamma_\text{opt}(\alpha,\beta,\lambda_\text{min}^A,\lambda_\text{max}^A)$.

In many applications $A$ is positive semidefinite with a non-trivial kernel. Then $\lambda_\text{min}^A = 0$ and $\lambda_\text{max}^A = \|a\|$, which leads to the cubic equation
\[
  \mu^3 - \|a\| \, \mu^2 - \beta^2 \, \mu + \alpha \, \beta^2 = 0 ,
\]
whose smallest positive root is a sharp lower bound for the positive eigenvalues 
$\mu$ of $\mathcal{P}^{-1} \mathcal{M}$ in this case, see Theorem~\ref{thmsym} and 
the example above.
The second part of Theorem~\ref{thmsym} yields the following simpler bound:
\[
  \mu 
   \ge \frac{\beta}{2 \|a\|} \left(- \beta + 
      \sqrt{ \beta^2 + 4  \alpha  \|a\|} \right)
    =  \frac{2 \alpha \beta }{\beta + \sqrt{ \beta^2 + 4  \alpha  \|a\|}}
\]
for the positive eigenvalues $\mu$ of $\mathcal{P}^{-1} \mathcal{M}$.

Finally, we would like to stress that
because of simple monotonicity arguments,
all presented estimates on the spectrum of $\mathcal{P}^{-1} \mathcal{M}$ remain valid if $\alpha$, $\beta$, $\gamma$, and $\lambda_\text{min}^A$ are replaced by lower bounds and $\lambda_\text{max}^A$, $\|b\|$, and $\|\mathcal{B}\|$ are replaced by upper bounds. For example, if $\lambda_\text{min}^A$ is replaced by its lower bound $-\|a\|$ and $\lambda_\text{max}^A$ by its upper bound $\|a\|$ in (\ref{cubic_sym}), then we also obtain a lower bound for the positive eigenvalues of $\mathcal{P}^{-1}\mathcal{M}$, which is not necessarily sharp but it goes without specific knowledge on the spectrum of $A$ other than the spectral radius of $P^{-1} A$. Observe that, with these replacements, (\ref{cubic_sym}) coincides with (\ref{cubic}), and the corresponding smallest positive root is identical to $\gamma_\text{opt}(\alpha,\beta,\|a\|)$,  derived in Sect.~\ref{general}.

\section{On a class of matrices with symmetric spectrum}
\label{symspectrum}

In this section we specialize our considerations to the 
following Hermitian matrix
$$
{\cal M} = 
\begin{bmatrix}
A & B^* \\ B & -A
\end{bmatrix} \in \CC^{2n\times 2n} ,
$$
with $A\in\RR^{n\times n}$ real and symmetric positive definite,
and $B \in \CC^{n\times n}$ complex symmetric, i.e., $B = B^T$, where
$B^T$ denotes the transpose of the possibly complex matrix $B$.
We know that the matrix $\cal M$ has all {\it real}
eigenvalues, $n$ positive and $n$ negative ones.
We next show that the negative eigenvalues are the
mirrored images of the positive eigenvalues.
This property has a few consequences, both in the choice
of the preconditioner, and on the convergence of MINRES.
Indeed, the spectrum of $\cal M$
is symmetric with respect to the origin, and MINRES behaves
like CG on a matrix having only the positive eigenvalues,
but with twice as many iterations. 
Therefore, MINRES on $\cal M$ will only make some 
progress every other iteration, showing complete stagnation
otherwise; we refer to \cite{Fischer.Ramageetal.98} for a similar
phenomenon for $2\times 2$ block matrices with a different nonzero structure.

We first need the following technical lemma.

\begin{lemma}\label{lemma:skewsym}
Let $H$ be a nonsingular complex symmetric matrix (i.e., $H=H^T$),
and $S$ be a complex skew-symmetric matrix (i.e., $S=-S^T$), both
of size $n$. Then the (nonzero) eigenvalues of the $2n\times 2n$ matrix
$$
\begin{bmatrix} 0 & I \\ H & S \end{bmatrix}
$$
come in pairs, $(\mu, -\mu)$.
\end{lemma}

\begin{proof}
For nonsingular $H$, we have the similarity transformation
$$
\begin{bmatrix} -i H^{\frac 1 2} &   \\   & I \end{bmatrix} 
\begin{bmatrix} 0 & I \\ H & S \end{bmatrix} 
\begin{bmatrix} i H^{-\frac 1 2} &   \\   & I \end{bmatrix} 
=
\begin{bmatrix} 0 & -i H^{\frac 1 2} \\ i H^{\frac 1 2} & S \end{bmatrix},
$$
where the  matrix in the right-hand side is complex skew-symmetric.
Since any skew-symmetric matrix $T$ is similar to its transpose $T^T=-T$,
to each Jordan block with eigenvalue $\mu$ in the Jordan
decomposition of $T$ there is
a corresponding Jordan block of the same size with eigenvalue 
$-\mu$.
\end{proof}

\begin{theorem} \label{SymEigen}
Assume that $B$ is nonsingular. Then the
eigenvalues $\mu$ of ${\cal M}$ come in pairs, $(\mu, -\mu)$,
with  $\mu \in \RR$.
\end{theorem}

\begin{proof}
The eigenvalue problem for $\cal M$ can be written as
$$
A x + B^* y = \mu x, \qquad B x - A y = \mu y.
$$
Substituting $x=B^{-1} (\mu I + A)y$ in the first equation
and after some simple algebra we obtain
$$
-\mu^2 B^{-1} y + \mu [A B^{-1}- B^{-1}A] y +
[A B^{-1}A + B^*] y = 0 .
$$
After multiplication by $B^{\frac 1 2}$ from both matrix sides we obtain
\begin{eqnarray*}
-\mu^2 \tilde y && + \mu \left [ B^{\frac 1 2}A B^{-\frac 1 2}- B^{-\frac 1 2}A
 B ^{\frac 1 2}\right ] \tilde y 
 + \left [ B^{\frac 1 2}A B^{-1}A B^{\frac 1 2} + 
 B^{\frac 1 2}B^* B^{\frac 1 2}\right ] \tilde y = 0 ,
\end{eqnarray*}
with $\tilde y =  B^{-\frac 1 2} y$, which gives
$$
\left [ -\mu^2 I + \mu (G-G^T)  + 
(GG^T + B^{\frac 1 2}B^*B^{\frac 1 2}) \right ] \tilde y = 0.
$$
The eigenvalues of the quadratic matrix equation above can be obtained
as eigenvalues of the linearized problem
$$
\begin{bmatrix}
O & I \\ GG^T + B^{\frac 1 2}B^* B^{\frac 1 2} & G-G^T
\end{bmatrix}
z = \mu z \quad \Leftrightarrow \quad
{\cal G} z = \mu z .
$$
We then recall that all eigenvalues of 
${\cal M}$, and thus of ${\cal G}$,
 are real, and that
$n$ of them are positive and the other $n$ are negative.
The matrix ${\cal G}$ satisfies the hypotheses of
Lemma \ref{lemma:skewsym}, therefore the eigenvalues of ${\cal G}$
come in pairs $(\mu, -\mu)$, which completes the proof.
\end{proof}

\section{Application to optimal control problems}
\label{applications}

Let $\Omega$ be an open and bounded domain in $\mathbb{R}^d$ for $d \in \{1,2,3\}$ with Lipschitz-con\-tin\-u\-ous boundary $\Gamma$. For $T > 0$, we introduce the space-time cylinder $Q_T = \Omega \times (0,T)$ and its lateral surface $\Sigma_T = \Gamma \times (0,T)$.

\subsection{Distributed optimal control for time-periodic parabolic equations}
\label{tp:parabolic}

First we consider the following model problem: Find the state $y(x,t)$ and the control $u(x,t)$ that minimizes the cost functional
\[
  J(y,u) = \frac{1}{2} \int_0^T \int_\Omega |y(x,t)-y_d(x,t)|^2 \ d x \ d t + 
      \frac{\nu}{2} \int_0^T \int_\Omega |u(x,t)|^2 \ d x \ d t 
\]
subject to the time-periodic parabolic problem
\begin{alignat*}{2}
  \frac{\partial}{\partial t} y(x,t) - \Delta y(x,t) & \ = \ u(x,t) & &\quad \text{in} \ Q_T, \\
  y(x,t) & \ = \ 0  & &\quad \text{on} \ \Sigma_T, \\
  y(x,0) & \ = \ y(x,T) & & \quad \text{on} \ \Omega, \\
  u(x,0) & \ = \ u(x,T) & & \quad \text{on} \ \Omega.
\end{alignat*}
Here $y_d(x,t)$ is a given target (or desired) state and $\nu > 0$ is a cost or regularization parameter. 
We assume that $y_d(x,t)$ is time-harmonic, i.e.:
\[
  y_d(x,t) = y_d(x) \, e^{i \omega t}
  \quad \text{with} \quad \omega = \frac{2\pi k}{T} \quad \text{for some} \ k \in \mathbb{Z}.
\]
Then there is a time-periodic solution to the original control problem of the form
\[
  y(x,t) = y(x) \, e^{i \omega t}, \quad
  u(x,t) = u(x) \, e^{i \omega t},
\]
where $y(x)$ and $u(x)$ solve the following time-independent optimal control problem: 
Minimize
\[
  \frac{1}{2} \int_\Omega |y(x) - y_d(x)|^2 \ d x + 
      \frac{\nu}{2} \int_\Omega |u(x)|^2 \ d x 
\]
subject to
\begin{alignat*}{2}
   i \omega \, y(x) - \Delta y(x) & \ = \ u(x) & &\quad \text{in} \ \Omega, \\
   y(x) & \ = \ 0  & &\quad \text{on} \ \Gamma.
\end{alignat*}
Using an appropriate finite element space $V_h$ of dimension $n$ for 
both $y$ and $u$, we obtain the following discrete version: 
Minimize
\[
  \frac{1}{2} (\underline{y} - \underline{y}_d)^* M (\underline{y} - \underline{y}_d) + 
      \frac{\nu}{2} \, \underline{u}^* M  \underline{u}
\]
subject to
\begin{alignat*}{2}
  i \omega \, M \, \underline{y} + K \underline{y} & \ = M \underline{u}.
\end{alignat*}
Here the real matrices $M$ and $K$ are the mass matrix, representing the $L^2$-inner product in $V_h$, and the discretized negative Laplacian, respectively. The underlined quantities $\underline{y}$, $\underline{y}_d$, and $\underline{u}$ denote the coefficient vectors of the corresponding finite element functions relative to the chosen set of basis functions in $V_h$. 

The Lagrangian functional for this constrained optimization problem is given by
\begin{align*}
  \mathcal{L}(\underline{y},\underline{u},\underline{p}) & = \frac{1}{2} (\underline{y} - \underline{y}_d)^* M (\underline{y} - \underline{y}_d) + 
      \frac{\nu}{2} \, \underline{u}^* M  \underline{u}
  + \underline{p}^* \left(i \omega \, M \, \underline{y} + K \underline{y} - M \underline{u} \right),
\end{align*}
where  $\underline{p}$ denotes the Lagrangian multiplier associated with the constraint.
The first-order necessary optimality conditions, which are also sufficient for the problem considered here, are $\nabla \mathcal{L}(\underline{y}, \underline{u},\underline{p}) = 0$, and read in details:
\begin{equation} \label{KKT}
  \begin{bmatrix}
  M  & 0 &  K - i \omega \, M \\
  0 & \nu \, M & - M \\
  K + i \omega \, M & -M & 0
  \end{bmatrix}
  \begin{bmatrix} \underline{y} \\ \underline{u} \\ \underline{p} \end{bmatrix}
   =
  \begin{bmatrix} M \underline{y}_d \\ 0 \\ 0 \end{bmatrix} .
\end{equation}
This is a linear system of the form (\ref{M}) with 
\begin{equation} \label{ABparabol3}
   A = \begin{bmatrix} M & 0 \\ 0 & \nu \, M \end{bmatrix}  \in \mathbb{R}^{2n \times 2n}, \quad 
   B = \begin{bmatrix} K + i \omega \, M & - M \end{bmatrix} \in \mathbb{C}^{n \times 2n}, \quad 
   C = 0 .
\end{equation}
The system (\ref{KKT}) was discussed in \cite{schoeberl:07b} for the special case $\omega = 0$, which corresponds to an elliptic optimal control problem. Observe that in this case all matrices are real. The preconditioner constructed in \cite{schoeberl:07b} is an indefinite 3-by-3 block matrix and
leads to convergence rates of the preconditioned conjugate gradient method which do not deteriorate if the mesh size, say $h$, and/or the cost parameter $\nu$ approach 0. 

Based on ideas very close to those developed in \cite{schoeberl:07b} we obtain the following properties of the bilinear forms $a$ and $b$ associated with $A$ and $B$, respectively, in terms of the quantities $\alpha$, $\lambda_\text{min}^A$, $\lambda_\text{max}^A$, $\beta$, and $\|b\|$:

\begin{theorem} \label{CoeffEx1}
Let $\mathcal{M}$ be given by (\ref{M}) and (\ref{ABparabol3}). Then, for $\mathcal{P}$, given by
\[
  \mathcal{P} = \begin{bmatrix} P & 0 \\ 0 & R \end{bmatrix} 
  \ \text{with} \ 
  P = \begin{bmatrix} Y & 0 \\ 0 & \nu \, M\end{bmatrix}, 
  \ 
  R = \frac{1}{\nu} \, Y,
  \  \text{and} \  
  Y = M + \sqrt{\nu} \, (K + \omega \, M) ,
\]
we have
\[
  \alpha \ge 2 - \sqrt{2}, \quad 
  \lambda_\text{min}^A \ge 0, \quad
  \lambda_\text{max}^A \le 1, \quad
  \beta \ge \frac{\sqrt{2}}{2}, \quad \|b\| \le 1.
\]
\end{theorem}

\begin{proof}
The estimate $\lambda_\text{min}^A \ge 0$ is trivial. The upper bounds for $\lambda_\text{max}^A = \|a\|$ and $\|b\|$ follow completely analoguously to the proof of Lemma 4.1 in \cite{schoeberl:07b}. 

The proof of the lower bounds for the inf-sup constants $\alpha$ and $\beta$ in \cite{schoeberl:07b} covers only the case $\omega = 0$. An essential step of that proof was the estimate 
\[
  \|y\|_K^2 \le \|y\|_M\, \|u\|_M
\]
for all $y$ and $u$ satisfying the state equation for $\omega = 0$, i.e., $K y = M u$. 
That part of the proof has to be replaced for general $\omega$ by the estimate 
\[
  \|y\|_K^2 + \omega \, \|y\|_M^2 
     \le \sqrt{2} \, \|y\|_M \|u\|_M
\]
for all $y$ and $u$ satisfying the state equation $(K + i \omega \, M) y = M u$, which easily follows:
\begin{align*}
  \left( \|y\|_K^2 + \omega \, \|y\|_M^2 \right)^2  
   & \le 2 \, \left( \|y\|_K^4 + \omega^2 \, \|y\|_M^4 \right) \\
   &  =  2 \, \left| \langle (K + i \omega \, M) y, y \rangle \right|^2 
      =  2 \, \left| \langle Mu , y \rangle \right|^2 
     \le 2 \, \|y\|_M^2 \|u\|_M^2 .
\end{align*}
All other arguments are completely identical to the corresponding arguments used in \cite{schoeberl:07b} and are omitted. 
\end{proof}

If $\alpha$, $\lambda_\text{min}^A$, $\lambda_\text{max}^A$, $\beta$, and $\|b\|$ in (\ref{mu124}) and (\ref{cubic_sym}) are replaced by the corresponding lower or upper bounds provided by Theorem \ref{CoeffEx1}, it immediately follows that the spectrum of $\mathcal{P}^{-1} \mathcal{M}$ is contained in the set
\[
  \left[-1, \frac{1}{2}(1 - \sqrt{3}) \right] 
  \cup
  \left[ \mu_3, \frac{1}{2}(1 + \sqrt{5}) \right],
\]
where $\mu_3$ is the smallest positive root of the cubic equation
\[
  \mu^3 - \mu^2 - \frac{1}{2} \, \mu + 1 - \frac{\sqrt{2}}{2} = 0 .
\]
These intervals read in 3-digit accuracy
\[
  [-1, -0.366] \cup [0.396, 1.618].
\]
From the second part of Theorem \ref{thmsym} we know a simple lower bound for $\mu_3$:
\[
   \mu_3 \ge \frac{4 - 2 \sqrt{2}}{1 + \sqrt{17 - 8 \sqrt{2}}} \approx 0.346.
\]

There is a frequently used alternative approach for solving (\ref{KKT}). 
From the second row of (\ref{KKT}) it follows that $\underline{u} = \underline{p}/\nu$. Therefore, we can eliminate the control $\underline{u}$ and obtain the reduced optimality system
\[
  \begin{bmatrix}
  M  & K - i \omega \, M \\
  K + i \omega \, M & - \frac{1}{\nu} M
  \end{bmatrix}
  \begin{bmatrix} \underline{y} \\ \underline{p} \end{bmatrix}
   =
  \begin{bmatrix} M \underline{y}_d \\ 0 \end{bmatrix} .
\]
This system was discussed in \cite{zulehner:11} for the special case $\omega = 0$.
The preconditioner constructed in \cite{zulehner:11} leads to convergence rates of the preconditioned MINRES method which do not deteriorate if the mesh size $h$ and/or the cost parameter $\nu$ approach 0. The results from \cite{zulehner:11} were extended to the case $\omega \neq 0$ in \cite{kolmbauer:11} (based on results from \cite{copeland:11}), where a preconditioner was constructed and a bound for the number of MINRES-iterations was derived which is independent of $h$, $\nu$, and $\omega$.

Here we will shed some new light on the preconditioner from \cite{kolmbauer:11} 
by presenting a slightly different analysis of the already known properties within the framework of complex matrices, 
and by supplementing these properties by a new statement on the symmetry of the spectrum of the preconditioned matrix.
This new approach is helpful for extending the analysis of preconditioners for other optimal control problems, like the one discussed in the subsequent subsection.
 
A simple scaling leads to the equivalent system
\[
  \begin{bmatrix}
  M & \sqrt{\nu} \, (K - i \omega \, M) \\[1ex]
  \sqrt{\nu} \, (K + i \omega \, M) & - M 
  \end{bmatrix}
  \begin{bmatrix} \underline{y} \\[1ex] \frac{1}{\sqrt{\nu}} \, \underline{p}
  \end{bmatrix}
   =
  \begin{bmatrix} M \underline{y}_d \\[1ex] 0 \end{bmatrix} ,
\]
which is of the form (\ref{M}) with
\begin{equation} \label{ABparabol}
   A = M  \in \mathbb{R}^{n \times n}, \quad
   B = \sqrt{\nu} \, (K + i \omega \, M) \in \mathbb{C}^{n \times n},
   \  \text{and} \quad
   C = -A .
\end{equation}
Using Sect.~\ref{general} and \ref{symspectrum} we obtain the following results on the eigenvalues of $\mathcal{P}^{-1} \mathcal{M}$:

\begin{theorem}
Let $\mathcal{M}$ be given by (\ref{M}) and (\ref{ABparabol}). 
Then, for $\mathcal{P}$, given by
\[
  \mathcal{P} = \begin{bmatrix} P & 0 \\ 0 & P \end{bmatrix} ,
\]
with $P$ real and symmetric positive definite,
the spectrum of $\mathcal{P}^{-1} \mathcal{M}$ is real and symmetric around zero. 
Moreover,
for $P = M + \sqrt{\nu} \, (K + \omega \, M)$, 
the following estimates hold 
\[
  \gamma \ge \frac{1}{\sqrt{3}}
  \quad \text{and} \quad 
  \|\mathcal{B}\| \le 1 ,
\]
where $\gamma$ is
the inf-sup constant, and 
$\mathcal{B}$ is the bilinear form associated with $\mathcal{M}$.
\end{theorem}

\begin{proof}
The symmetry of the spectrum of $\mathcal{P}^{-1} \mathcal{M}$ around zero directly follows from Theorem~\ref{SymEigen} applied to the similar matrix $\mathcal{P}^{-\frac{1}{2}} \mathcal{M}\mathcal{P}^{-\frac{1}{2}}$.
Let
\[
  \mathcal{H} = \begin{bmatrix} I & (1 - i) \, I \\ (1+i) \, I & - I \end{bmatrix} .
\]
Then, by direct calculations, one shows that
\[
  \Re \langle \mathcal{M} x, \mathcal{H} x \rangle = \|x\|_\mathcal{P}^2 
  \quad \text{and} \quad \|\mathcal{H} x \|_\mathcal{P}^2 = 3  \, \|x\|_\mathcal{P}^2
  \quad \text{for all} \ x \in \mathbb{C}^{2n} .
\]
Therefore,
\begin{align*}
  \gamma
   & = \inf_{\rule[0.6ex]{0ex}{1ex} 0 \neq x \in X} \sup_{0 \neq w \in X} 
           \frac{|\langle \mathcal{M} x, w \rangle |}{\|x\|_\mathcal{P} \|w\|_\mathcal{P}} 
   \ge \inf_{\rule[0.6ex]{0ex}{1ex} 0 \neq x \in X}  
           \frac{|\langle \mathcal{M} x, \mathcal{H} x \rangle|}{\|x\|_\mathcal{P} \|\mathcal{H} x\|_\mathcal{P}} \\ 
   & \ge \inf_{\rule[0.6ex]{0ex}{1ex} 0 \neq x \in X}  
           \frac{\Re \langle \mathcal{M} x, \mathcal{H} x \rangle}{\|x\|_\mathcal{P} \|\mathcal{H} x\|_\mathcal{P}} 
     =  \inf_{\rule[0.6ex]{0ex}{1ex} 0 \neq x \in X}  
           \frac{\|x\|_\mathcal{P}^2}{\|x\|_\mathcal{P} \sqrt{3} \|x\|_\mathcal{P}} 
    = \frac{1}{\sqrt{3}} .
\end{align*}
Moreover, since
\begin{align*}
  \langle \mathcal{M} x,x \rangle
   & = \langle M y,y \rangle + 2 \sqrt{\nu} \big( \Re \langle K y, p \rangle - \omega  \Im \langle M y, p \rangle \big)
   - \langle M p, p \rangle 
   \  \text{for} \ x = \begin{bmatrix} y \\ p \end{bmatrix} ,
\end{align*}
and
\[
  |\Im \langle M y, p \rangle| \le |\langle M y, p \rangle | \le \|y\|_M \|p\|_M, \ 
  |\Re \langle K y, p \rangle| \le |\langle K y, p \rangle | \le \|y\|_K \|p\|_K,
\]
it follows that
\begin{align*}
  |\langle \mathcal{M} x,x \rangle| 
  & \le \|y\|_M^2  + 2 \sqrt{\nu} \left( \|y\|_K \|p\|_K + \omega \, \|y\|_M \|p\|_M \right) + \|p\|_M^2 \\
  & \le \|y\|_M^2  + \sqrt{\nu} \left( \|y\|_K^2 + \|p\|_K^2 + \omega \, \left(\|y\|_M^2 + \|p\|_M^2 \right) \right) + \|p\|_M^2 \\
  & = \langle \mathcal{P} x, x \rangle \quad \text{for all} \ x \in \mathbb{C}^{2n},
\end{align*}
which implies that 
\[
  \|\mathcal{B}\|
     = \sup_{0 \neq x \in X} \sup_{0 \neq w \in X} 
           \frac{|\langle \mathcal{M} x, w \rangle |}{\|x\|_\mathcal{P} \|w\|_\mathcal{P}} 
     =  \sup_{0 \neq x \in X}  
           \frac{|\langle \mathcal{M} x, x \rangle|}{\langle \mathcal{P} x, x \rangle} \\ 
   \le 1. 
\]
This completes the proof. 
\end{proof}

\vskip 0.1in
If $\gamma$ is replaced by the lower bound $1/\sqrt{3}$ and $\|\mathcal{B}\|$ by the upper bound 1, it immediately follows from (\ref{symbound}) that the spectrum of $\mathcal{P}^{-1} \mathcal{M}$ is contained in the set
\[
  \left[-1, - \frac{1}{\sqrt{3}} \right] \cup
  \left[ \frac{1}{\sqrt{3}}, 1 \right]
   \approx [-1, -0.577] \cup [0.557,1].
\]

\begin{remark}
{\rm
The essential step for estimating the inf-sup constant $\gamma$ from below was the introduction of the matrix $\mathcal{H}$. Translated into a general Hilbert space setting and slightly more general,  the essential requirements on such a linear operator $\mathcal{H} \colon X \longrightarrow X$ are that there are some positive constants $c_1$, $c_2$ such that
\[
  |\mathcal{B}(z,\mathcal{H} z)| \ge c_1 \, \|z\|_X^2 
  \quad \text{and} \quad
  \|\mathcal{H} z\| \le c_2 \, \|z\|_X
  \quad \text{for all} \quad z \in X .
\]
Then it follows analogously to the previous proof that $\gamma \ge c_1 / c_2$.
}
\end{remark}

\subsection{Distributed optimal control for the time-periodic Stokes equations}
\label{tp:Stokes}

Next we consider the following problem: Find the velocity $\mathbf{u}(x,t)$, the pressure $p(x,t)$, and the force $\mathbf{f}(x,t)$ that minimizes the cost functional
\[
  J(\mathbf{u},\mathbf{f}) = \frac{1}{2} \int_0^T \int_\Omega |\mathbf{u}(x,t)-\mathbf{u}_d(x,t)|^2 \ d x \ d t + 
      \frac{\nu}{2} \int_0^T \int_\Omega |\mathbf{f}(x,t)|^2 \ d x \ d t 
\]
subject to the time-periodic Stokes problem
\begin{alignat*}{2}
  \frac{\partial}{\partial t} \mathbf{u}(x,t) - \Delta \mathbf{u}(x,t) + \nabla p(x,t) & \ = \ \mathbf{f}(x,t) & &\quad \text{in} \ Q_T, \\
  \nabla \cdot \mathbf{u}(x,t) & \ = \ 0 & &\quad \text{in} \ Q_T, \\
  \mathbf{u}(x,t) & \ = \ 0  & &\quad \text{on} \ \Sigma_T, \\
  \mathbf{u}(x,0) & \ = \ \mathbf{u}(x,T) & & \quad \text{on} \ \Omega, \\
  p(x,0) & \ = \ p(x,T) & & \quad \text{on} \ \Omega, \\
  \mathbf{f}(x,0) & \ = \ \mathbf{f}(x,T) & & \quad \text{on} \ \Omega.
\end{alignat*}
Here $\mathbf{u}_d(x,t)$ is a given target velocity, $\nu > 0$ is a cost or regularization parameter, and $|.|$ denotes the Euclidean norm in $\mathbb{R}^d$.  
We assume that $u_d(x,t)$ is time-harmonic, i.e.:
\[
  \mathbf{u}_d(x,t) = \mathbf{u}_d(x) \, e^{i \omega t}
  \quad \text{with} \quad \omega = \frac{2\pi k}{T} \quad \text{for some} \ k \in \mathbb{Z}.
\]
Then there is a time-periodic solution to the original control problem of the form
\[
  \mathbf{u}(x,t) = \mathbf{u}(x) \, e^{i \omega t}, \quad
  p(x,t) = p(x) \, e^{i \omega t}, \quad
  \mathbf{f}(x,t) = \mathbf{f}(x) \, e^{i \omega t},
\]
where $\mathbf{u}(x)$, $p(x)$, and $\mathbf{f}(x)$ solve the following time-independent optimal control problem: 
Minimize
\[
  \frac{1}{2} \int_\Omega |\mathbf{u}(x) - \mathbf{u}_d(x)|^2 \ d x + 
      \frac{\nu}{2} \int_\Omega |\mathbf{f}(x)|^2 \ d x 
\]
subject to
\begin{alignat*}{2}
   i \omega \, \mathbf{M} \mathbf{u}(x) - \Delta \mathbf{u}(x) + \nabla p(x) & \ = \ \mathbf{f}(x) & &\quad \text{in} \ \Omega, \\
  \nabla \cdot \mathbf{u}(x) & \ = \ 0 & &\quad \text{in} \ \Omega, \\
  \mathbf{u}(x) & \ = \ 0  & &\quad \text{on} \ \Gamma.
\end{alignat*}
Using appropriate finite element spaces $\mathbf{V}_h$ of dimension $n$ and $Q_h$ of dimension $m$ for $\mathbf{u}$ and $p$, respectively, and the same finite element space $\mathbf{V}_h$ for $\mathbf{f}$ as well, we obtain the following discrete version: 
Minimize
\[
  \frac{1}{2} (\underline{\mathbf{u}} - \underline{\mathbf{u}}_d)^* \mathbf{M} (\underline{\mathbf{u}} - \underline{\mathbf{u}}_d) + 
      \frac{\nu}{2} \, \underline{\mathbf{f}}^* \mathbf{M}  \underline{\mathbf{f}}
\]
subject to
\begin{alignat*}{2}
  i \omega \, \mathbf{M}\, \underline{\mathbf{u}} + \mathbf{K} \underline{\mathbf{u}} - \mathbf{D}^T \underline{p} & \ = \mathbf{M} \underline{\mathbf{f}} , \\
  \mathbf{D} \underline{\mathbf{u}} & \ = \ 0 .
\end{alignat*}
Here the real matrices $\mathbf{M}$, $\mathbf{K}$, and $\mathbf{D}$ are the mass matrix, representing the $L^2$-inner product in $\mathbf{V}_h$, the discretized negative vector Laplacian, and the discretized divergence, respectively. 
As before, underlined quantities denote the coefficient vectors of finite element functions relative to a basis. 

Completely similar to the discussion in the previous subsection we obtain the following reduced optimality system, again after eliminating the control, here $\underline{\mathbf{f}}$.
\[
  \begin{bmatrix}
  \mathbf{M} & 0 & \mathbf{K} - i \omega \, \mathbf{M} & - \mathbf{D}^T \\
   0 & 0 & - \mathbf{D} & 0 \\
   \mathbf{K} + i \omega \, \mathbf{M} & - \mathbf{D}^T & - \frac{1}{\nu} \mathbf{M} & 0 \\
   - \mathbf{D} & 0 & 0 & 0
  \end{bmatrix}
  \begin{bmatrix} \underline{\mathbf{u}} \\ \underline{p} \\ \underline{\mathbf{w}} \\ \underline{r} \end{bmatrix}
   =
  \begin{bmatrix} \mathbf{M} \underline{\mathbf{u}}_d \\ 0 \\ 0 \\ 0 \end{bmatrix}
\]
This system was discussed in \cite{zulehner:11} for the special case $\omega = 0$, which corresponds to the steady-state version of the control problem.
The preconditioner constructed in \cite{zulehner:11} leads to convergence rates of the preconditioned MINRES method 
that do not deteriorate if the mesh size $h$ and/or the cost parameter $\nu$ approach 0.

We will now apply the theoretical findings of the preceding sections and construct a preconditioner that will also work for general $\omega$. Convergence results will be derived that guarantee a bound for the number of iterations which is independent of $h$, $\nu$, and $\omega$. If applied to the special case $\omega = 0$, this bound is more accurate than the bound derived in \cite{zulehner:11}.

A first and essential observation is that, by swapping the second and the third rows and columns, we obtain a saddle point matrix with a vanishing 2-by-2 block in the right lower part. This leads with a simple scaling to the system
\[
  \begin{bmatrix}
  \mathbf{M} & \sqrt{\nu} \, (\mathbf{K} - i \omega \, \mathbf{M}) & 0 & - \sqrt{\nu} \, \mathbf{D}^T \\[1ex]
  \sqrt{\nu} \, (\mathbf{K} + i \omega \, \mathbf{M}) & - \mathbf{M} & - \sqrt{\nu} \, \mathbf{D}^T & 0 \\[1ex]
  0 & - \sqrt{\nu} \, \mathbf{D} & 0 & 0 \\[1ex]
  - \sqrt{\nu} \, \mathbf{D} & 0 & 0 & 0
  \end{bmatrix}
  \begin{bmatrix} \underline{\mathbf{u}} \\[1ex] \frac{1}{\sqrt{\nu}} \, \underline{\mathbf{w}} \\[1ex] \underline{p}  \\[1ex] \frac{1}{\sqrt{\nu}} \, \underline{r} \end{bmatrix}
   =
  \begin{bmatrix} \mathbf{M} \underline{\mathbf{u}}_d \\[1ex] 0 \\[1ex] 0 \\[1ex] 0 \end{bmatrix} ,
\]
which is of the form (\ref{M}) with 
\begin{equation} \label{AStokes}
   A = \begin{bmatrix} \mathbf{M} & \sqrt{\nu} \,  (\mathbf{K} - i \omega \, \mathbf{M}) \\ \sqrt{\nu} \, (\mathbf{K} + i \omega \, \mathbf{M}) & - \mathbf{M} \end{bmatrix} \in \mathbb{C}^{2n \times 2n},
\end{equation}
\begin{equation} \label{BStokes}
   B = - \sqrt{\nu} \, \begin{bmatrix} 0 & \mathbf{D} \\ \mathbf{D} & 0 \end{bmatrix} \in \mathbb{C}^{2m \times 2n} ,
  \quad \text{and} \quad
   C = 0.
\end{equation}
Using the results of Sect.~\ref{Hermitian} and \ref{symspectrum} we obtain the following properties of the bilinear forms $a$ and $b$ associated with $A$ and $B$, respectively, in terms of the quantities $\alpha$, $\|a\|$, $\beta$, and $\|b\|$:
\begin{theorem} \label{CoeffEx2}
Let $\mathcal{M}$ be given by (\ref{M}), (\ref{AStokes}), and (\ref{BStokes}). 
Then, for $\mathcal{P}$, given by
\[
  \mathcal{P} = \begin{bmatrix} P & 0 \\ 0 & R \end{bmatrix}
  \quad \text{with} \quad 
  P = \begin{bmatrix} \mathbf{P} & 0 \\
   0 & \mathbf{P}
  \end{bmatrix} 
  \quad  \text{and} \quad 
  R =  \nu \begin{bmatrix} S & 0 \\
   0 & S 
  \end{bmatrix},
\]
with $\mathbf{P}$ and $S$ real and symmetric positive definite,
the spectrum of $\mathcal{P}^{-1} \mathcal{M}$ is real and symmetric around zero. 
Moreover, for
\begin{equation} \label{PS}
  \mathbf{P} =  \mathbf{M} + \sqrt{\nu} \, (\mathbf{K} + \omega \, \mathbf{M})
  \quad \text{and} \quad
  S = \mathbf{D} \left[ \mathbf{M} + \sqrt{\nu} \, (\mathbf{K} + \omega \, \mathbf{M}) \right]^{-1} \mathbf{D}^T,
\end{equation}
the following estimates hold
\[
  \alpha \ge \frac{1}{\sqrt{3}}, \quad 
  \|a\| \le 1, \quad
  \beta = 1, \quad \|b\| = 1.
\]
\end{theorem}

\begin{proof}
The symmetry of the spectrum around zero follows from Theorem~\ref{SymEigen} applied to the system in the original ordering of the rows and columns.
Let
\[
  H = \begin{bmatrix} \mathbf{I} & (1-i) \, \mathbf{I} \\ (1+i) \, \mathbf{I} & - \mathbf{I}  \end{bmatrix} .
\]
Then, by direct calculations, one shows that
\[
  \Re \langle A u, H u \rangle = \|u\|_P^2 
  \quad \text{and} \quad \|H u \|_P^2 = 3  \, \|u\|_P^2
  \quad \text{for all} \ u \in \mathbb{C}^{2n} .
\]
Furthermore, observe that 
\[
  H u \in \ker B \quad \text{for all} \quad u \in \ker B.
\]
Therefore,
\begin{align*}
  \alpha 
   & = \inf_{\rule[0.6ex]{0ex}{1ex} 0 \neq u \in \ker B} \sup_{0 \neq v \in \ker B} 
           \frac{|\langle A u, v \rangle |}{\|u\|_P \|v\|_P} 
   \ge \inf_{\rule[0.6ex]{0ex}{1ex} 0 \neq u \in \ker B}  
           \frac{|\langle A u, H u \rangle|}{\|u\|_P \|Hu\|_P} \\ 
   & \ge \inf_{\rule[0.6ex]{0ex}{1ex} 0 \neq u \in \ker B}  
           \frac{\Re \langle A u, H u \rangle}{\|u\|_P \|Hu\|_P} 
     =  \inf_{\rule[0.6ex]{0ex}{1ex} 0 \neq u \in \ker B}  
           \frac{\|u\|_P^2}{\|u\|_P \sqrt{3} \|u\|_P} 
    = \frac{1}{\sqrt{3}} .
\end{align*}
Moreover, since
\begin{align*}
  & \langle A u,u \rangle = \langle \mathbf{M} \underline{\mathbf{u}},\underline{\mathbf{u}} \rangle \\
  & \quad + 2 \sqrt{\nu} \big( \Re \langle \mathbf{K} \underline{\mathbf{u}}, \underline{\mathbf{w}} \rangle - \omega \, \Im \langle \mathbf{M} \underline{\mathbf{u}},\underline{\mathbf{w}} \rangle \big)
   - \langle \mathbf{M} \underline{\mathbf{w}}, \underline{\mathbf{w}} \rangle
  \ \text{with} \ 
  u = \begin{bmatrix} \underline{\mathbf{u}} \\ \underline{\mathbf{w}} \end{bmatrix}
\end{align*}
and
\begin{align*}
  |\Im \langle \mathbf{M} \underline{\mathbf{u}}, \underline{\mathbf{w}} \rangle| 
    & \le |\langle \mathbf{M} \underline{\mathbf{u}}, \underline{\mathbf{w}} \rangle | \le \|\underline{\mathbf{u}}\|_\mathbf{M} \|\underline{\mathbf{w}}\|_\mathbf{M}, \\
  |\Re \langle \mathbf{K} \underline{\mathbf{u}}, \underline{\mathbf{w}} \rangle| 
    & \le |\langle \mathbf{K} \underline{\mathbf{u}}, \underline{\mathbf{w}} \rangle | \le \|\underline{\mathbf{u}}\|_\mathbf{K} \|\underline{\mathbf{w}}\|_\mathbf{K},
\end{align*}
it follows that
\begin{align*}
  |\langle A u,u \rangle| 
  & \le \|\underline{\mathbf{u}}\|_\mathbf{M}^2  + 2 \sqrt{\nu} \left( \|\underline{\mathbf{u}}\|_\mathbf{K} \|\underline{\mathbf{w}}\|_\mathbf{K} + \omega \|\underline{\mathbf{u}}\|_\mathbf{M} \|\underline{\mathbf{w}}\|_\mathbf{M}\right) + \|\underline{\mathbf{w}}\|_\mathbf{M}^2 \\
  & \le \|\underline{\mathbf{u}}\|_\mathbf{M}^2  + \sqrt{\nu} \left( \|\underline{\mathbf{u}}\|_\mathbf{K}^2 + \|\underline{\mathbf{w}}\|_\mathbf{K}^2 + \omega \left(\|\underline{\mathbf{u}}\|_\mathbf{M}^2 + \|\underline{\mathbf{w}}\|_\mathbf{M}^2 \right) \right) + \|\underline{\mathbf{w}}\|_\mathbf{M}^2 \\
  & = \langle P u, u \rangle \quad \text{for all} \ u \in \mathbb{C}^{2n},
\end{align*}
which implies that $\|a\| \le 1$.

Finally, since $\langle B P^{-1} B^T q,q \rangle = \langle R q,q \rangle$ for all $q \in \mathbb{C}^{2m}$,
it directly follows from (\ref{beta}) and (\ref{normb}) that $\beta = \|b\| = 1$. 
\end{proof}

If $\alpha$, $\|a\|$, $\beta$, and $\|b\|$ in (\ref{cubic}) and (\ref{upperbound}) are replaced by the corresponding lower or upper bounds provided by Theorem \ref{CoeffEx2}, it immediately follows from (\ref{symbound}) that the spectrum of $\mathcal{P}^{-1} \mathcal{M}$ is contained in the set
\[
  \left[-\frac{1}{2} 
      (1 + \sqrt{5}),  -\mu_3 \right] \cup
  \left[ \mu_3, \frac{1}{2} 
      (1 + \sqrt{5}) \right],
\]
where $\mu_3$ is the smallest positive root of the cubic equation
\[
  \mu^3 - 2 \mu + \frac{1}{\sqrt{3}} = 0 .
\]
These intervals read in 3-digit accuracy
\[
  [-1.618, -0.306] \cup [0.306, 1.618].
\]
From the second part of Theorem \ref{mr} we know a simple lower bound for $\mu_3$:
\[
   \mu_3 \ge \frac{1}{2\sqrt{3}} \approx 0.289 .
\]

\begin{remark}
{\rm
The essential step for estimating the inf-sup constant $\alpha$ from below was the existence of the matrix $H$. Translated into a general Hilbert space setting and slightly more general,  the essential requirements on such a linear operator $H \colon V \longrightarrow V$ are that
\begin{enumerate}
\item
there are some positive constants $c_1$, $c_2$ such that
\[
  |a(u,H u)| \ge c_1 \, \|u\|_V^2 
  \quad \text{and} \quad
  \|H u\|_V \le c_2 \, \|u\|_V
  \quad \text{for all} \quad u \in V ,
\]
and
\item
$\ker B$ is an invariant subspace of $H$.
\end{enumerate}
Then it follows analogously to the previous proof that $\alpha \ge c_1 / c_2$.
}
\end{remark}

\subsection{A numerical example}
\label{NumEx}

We refer to \cite{kolmbauer:11} for numerical experiments for time-periodic parabolic optimal control problems. 
Here we present some numerical experiments for the time-periodic Stokes control problem  on the unit square domain $\Omega = (0, 1) \times (0, 1) \subset \mathbb{R}^2$. 
Following Example 1 in \cite{gunzburger00} we choose the target velocity $\mathbf{u}_d(x,y) = \left[(U(x,y),V(x,y)\right]^T$, given by
\[
  U(x, y) = 10 \frac{\partial}{\partial y} (\varphi(x) \varphi(y)) 
 \quad \text{and} \quad  
  V(x, y) = - 10 \frac{\partial}{\partial x} (\varphi(x) \varphi(y))
\]
with
\[
  \varphi(z) = \big( 1 - \cos( 0.8 \pi z) \big) (1 - z)^2.
\]
The velocity $\mathbf{u}_d(x,y)$ is divergence free. Note that, contrary to the velocity tracking problem for time-periodic Stokes flow considered here, in \cite {gunzburger00} the velocity tracking problem was discussed for a time-dependent Navier-Stokes flow.  

The problem was discretized by the Taylor-Hood pair of finite element spaces consisting of continuous piecewise quadratic polynomials for the velocity $\mathbf{u}(x,y)$ and the force $\mathbf{f}(x,y)$ and continuous piecewise linear polynomials for the pressure $p(x,y)$ on a triangulation of $\Omega$. 
The initial mesh contains four triangles obtained by connecting the two diagonals.
The final mesh was constructed by applying $\ell$ uniform refinement steps to the initial mesh, leading to a mesh size $h = 
2^{-\ell}$. The total number of unknowns on the finest level $\ell = 4$ is 18 056.

Table~\ref{tab:1} - \ref{tab:3} contain the numerical results produced by the preconditioned MINRES method with the block diagonal preconditioner $\mathcal{P}$ as described in Theorem \ref{CoeffEx2}.
The considered values for the  mesh size $h$, the frequency $\omega$, and the regularization parameter $\nu$ are specified in the table caption and the first columns. 
The second columns show the minimal intervals $[\widehat{\mu}_3,\widehat{\mu}_4]$ that enclose all positive eigenvalues of $\mathcal{P}^{-1}\mathcal{M}$. These intervals were computed with an extended version of the preconditioned MINRES method by the help Ritz values and harmonic Ritz values, see \cite{silvester:11}.
The third columns show the (constant) interval $[\mu_3,\mu_4]$ containing all positive eigenvalues as discussed right after Theorem \ref{CoeffEx2} for comparison. 
The fourth columns contain the number $\hat{k}$ of MINRES iterations that are required for reducing the initial residual in the $\mathcal{P}^{-1}$-norm by a factor of $\varepsilon = 10^{-8}$ with initial guess $x_0 = 0$. 
This number $\hat{k}$ is compared with the theoretical bound $k$, which is shown in the last columns, based on the estimate
\[
  \|r_{2l} \|_{\mathcal{P}^{-1}} 
   \le \frac{2 q^l}{1 + q^{2l}} \, \|r_0\|_{\mathcal{P}^{-1}}
  \quad \text{with} \quad
  q = \frac{\kappa\left( \mathcal{P}^{-1} \mathcal{M}\right) - 1}{\kappa\left( \mathcal{P}^{-1} \mathcal{M}\right) + 1},
\]
for the residual $r_k$ of the $k$-th iterate, see, e.g., \cite{greenbaum:97}, and the eigenvalue estimates for $\mathcal{P}^{-1} \mathcal{M}$ discussed right after Theorem  \ref{CoeffEx2}.

\begin{table}[h]
\caption{$\nu = 1$, $\omega = 1$}
\label{tab:1}       

\medskip
\centering
\begin{tabular}{p{1.5cm}cc@{\quad\quad\quad}cc}
\hline\noalign{\smallskip}
$h$ & $[\widehat{\mu}_3,\widehat{\mu}_4]$ & $[\mu_3,\mu_4]$ & $\widehat{k}$ & $k$ \\
\noalign{\smallskip}\hline\noalign{\smallskip}
1      & [0.627 , 1.595] & [0.306, 1.618] & 6 & 102 \\
0.5    & [0.620 , 1.612] & [0.306, 1.618] & 26 & 102 \\
0.25   & [0.619 , 1.616] & [0.306, 1.618] & 28 & 102 \\
0.125  & [0.618 , 1.618] & [0.306, 1.618] & 28 & 102 \\
0.0625 & [0.618 , 1.618] & [0.306, 1.618] & 28 & 102 \\
\noalign{\smallskip}\hline
\end{tabular}
\end{table}

\begin{table}[h]
\caption{$h = 0.0625$, $\nu = 1$}
\label{tab:2}       

\medskip
\centering
\begin{tabular}{p{1.5cm}cc@{\quad\quad\quad}cc}
\hline\noalign{\smallskip}
$\omega$ & $[\widehat{\mu}_3,\widehat{\mu}_4]$ & $[\mu_3,\mu_4]$ & $\widehat{k}$ & $k$ \\
\noalign{\smallskip}\hline\noalign{\smallskip}
$0$    & [0.618 , 1.618] & [0.306, 1.618] & 18 & 102 \\
$1$    & [0.618 , 1.618] & [0.306, 1.618] & 28 & 102 \\
$10^2$ & [0.611 , 1.613] & [0.306, 1.618] & 42 & 102 \\
$10^4$ & [0.565 , 1.614] & [0.306, 1.618] & 44 & 102 \\
$10^8$ & [0.618 , 1.618] & [0.306, 1.618] & 16 & 102 \\
\noalign{\smallskip}\hline
\end{tabular}
\end{table}

\begin{table}[h]
\caption{$h = 0.0625$, $\omega = 1$}
\label{tab:3}       

\medskip
\centering
\begin{tabular}{p{1.5cm}cc@{\quad\quad\quad}cc}
\hline\noalign{\smallskip}
$\nu$ & $[\widehat{\mu}_3,\widehat{\mu}_4]$ & $[\mu_3,\mu_4]$ & $\widehat{k}$ & $k$ \\
\noalign{\smallskip}\hline\noalign{\smallskip}
$10^{-8}$ & [0.566 , 1.614] & [0.306, 1.618] & 43 & 102 \\
$10^{-4}$ & [0.611 , 1.613] & [0.306, 1.618] & 42 & 102 \\
$10^{-2}$ & [0.619 , 1.618] & [0.306, 1.618] & 38 & 102 \\
$1$       & [0.618 , 1.618] & [0.306, 1.618] & 28 & 102 \\
$10^8$    & [0.618 , 1.618] & [0.306, 1.618] & 28 & 102 \\
\noalign{\smallskip}\hline
\end{tabular}
\end{table}

As expected from the results of Theorem \ref{CoeffEx2}, the eigenvalues of $\mathcal{P}^{-1} \mathcal{M}$ are bound\-ed away from 0 and $\infty$ independent of $h$, $\omega$, and $\nu$ leading to a uniform bound for the number of MINRES-iterations. The theoretical bound $\mu_4$ for the largest positive eigenvalues is rather close to the observed values, while the theoretical (uniform) bound $\mu_3$ for the smallest positive eigenvalues underestimates the observed values. Numerical experiments indicate that this is mainly due to the lower bound $1/\sqrt{3}$, which underestimates $\alpha$, and not due to the subsequent use of the estimates from Theorem \ref{mr}.

\section{Concluding Remarks}
\label{Concluding}

The results from Sect.~2 - 4 apply to quite general classes of saddle point problems and provide sharp stability estimates as well as particular spectral properties. 
The subsequent discussion of two model problems from optimal control does not only demonstrate the applicability of the theoretical results but also shows the robust behavior of the associated preconditioned MINRES method with respect to 
the involved numerical and model parameters for the particularly chosen preconditioners. 

Concerning the implementation of the iterative methods one important issue has not been addressed so far. The use of a preconditioner $\mathcal{P}$ requires the evaluation of expressions of the form $\mathcal{P}^{-1} y$ for some given vector $y$. In all discussed optimality systems for the discretized model problems this is a nontrivial task due to the potentially high number of involved unknowns. For example, the application of the preconditioner for the velocity tracking problem requires the evaluation of
$\mathbf{P}^{-1} \underline{\mathbf{v}}$ and $S^{-1} \underline{q}$ with $\mathbf{P}$ and $S$, given by (\ref{PS}), for some vectors $\underline{\mathbf{v}}$ and $\underline{q}$. In a first step of approximation, $S^{-1}$ is replaced by $(1 + \sqrt{\nu} \omega) M_p^{-1} + \sqrt{\nu} \omega K_p^{-1}$ (Cahout-Charbard preconditioner, see \cite{cahouet:88}), where $M_p$ and $K_p$ are the mass matrix and the discretized negative Laplacian in the finite element space $Q_h$ for the pressure. With this replacement the application of the preconditioner involves only matrices which can be interpreted as discretized diffusion-reaction operators of second order. In a second step of approximation these matrices are replaced by efficient preconditioners (like multigrid preconditioners) which are well-established for this class of problems.
Such a modified and efficiently realizable preconditioner leads to similar performance results as the original theoretical preconditioner according to the analysis presented, e.g., in \cite{olshanskii:00}, \cite{bramble:97}, \cite{mardal:04}, \cite{mardal:06}, \cite{olshanskii:06}. 

\section*{Appendix}
\label{app}

The detailed arguments for the inequality (\ref{maxval}), see the proof of Theorem \ref{thmsym}, are as follows.

The matrix $\lambda_\text{max}^A \, P - A$ is positive semidefinite, see (\ref{lambda}). Therefore, its Schur complement is positive semidefinite, too:
\[
  \left\langle \left(
   \lambda_\text{max}^A \, P_1 - A_{11} - A_{10} \left(\lambda_\text{max}^A \, P_0 - A_{00} \right)^{-1} A_{01} \right) u_1, u_1 \right\rangle \ge 0 .
\]
Similar as in the proof of Theorem \ref{thmsym} one obtains
\[
  \left\langle \left(
   A_{10} \left(\lambda_\text{max}^A \, P_0 - A_{00} \right)^{-1} A_{01} \right) u_1, u_1 \right\rangle 
     \ge \frac{\alpha}{\lambda_\text{max}^A - \alpha}  \, 
     \left\langle A_{10} A_{00}^{-1} A_{01} u_1, u_1 \right\rangle .
\] 
This implies
\[
  \left\langle \left(\lambda_\text{max}^A \, P_1 - A_{11} \right) u_1, u_1 \right\rangle 
   - \frac{\alpha}{\lambda_\text{max}^A - \alpha}  \, \left\langle A_{10} A_{00}^{-1} A_{01} u_1, u_1 \right\rangle \ge 0,
\]
and, therefore, after dividing by $\langle P_1 u_1,u_1 \rangle$,
\begin{equation} \label{con1}
  \lambda_\text{max}^A - r_2  - \frac{\alpha }{\lambda_\text{max}^A - \alpha } \, r_1 \ge 0
\end{equation}
for the Rayleigh quotients $r_1$ and $r_2$ from the proof of Theorem \ref{thmsym}.

Moreover, the matrix $A - \lambda_\text{min}^A \, P$ is positive semidefinite. Then it follows analogously, however only for $\lambda_\text{min}^A \le 0$, that
\begin{equation} \label{con2}
  r_2 - \lambda_\text{min}^A  - \frac{\alpha }{\alpha - \lambda_\text{min}^A } \, r_1 \ge 0.
\end{equation}
In the proof of Theorem \ref{thmsym} we need an upper bound for
\[
  \phi(r_1,r_2) = \frac{\alpha}{\alpha - \mu} \,  r_1 - r_2.
\]
It is easy to see that the maximum of the linear function $\phi(r_1,r_2)$ under the linear constraints (\ref{con1}) and (\ref{con2}) is attained at the intersection of the corresponding straight lines.
By elementary calculations it follows that the value of $\phi(r_1,r_2)$ at the point of intersection of these lines is given by
\[
  \frac{(\lambda_\text{max}^A + \lambda_\text{min}^A - \alpha) \mu - \lambda_\text{max}^A \lambda_\text{min}^A} {\alpha - \mu},
\]
which completes the proof of (\ref{maxval}).

\bibliographystyle{spmpsci}      
\bibliography{spectprop}   

\end{document}